\documentclass[11pt]{amsart}
\usepackage{amssymb}
\usepackage{verbatim}
\usepackage{amscd}
\usepackage{amsmath}
\usepackage{easybmat} 
\usepackage{enumerate}
\usepackage{color}

\newtheorem{theorem}{Theorem}[section]
\newtheorem{corollary}[theorem]{Corollary}

\newtheorem{lemma}[theorem]{Lemma}

\theoremstyle{definition}
\newtheorem{Def}[theorem]{Definition}
\newtheorem{rmk}[theorem]{Remark}

\def\bfe{\mathbf{e}}
\def\bfq{\mathbf{q}}
\def\bfx{\mathbf{x}}
\def\bfy{\mathbf{y}}

\def\bfk{\mathbf{k}}
\def\bfn{\mathbf{n}}
\def\bfp{\mathbf{p}}

\def\bbz{{\mathbb{Z}}}

\def\bbr{{\mathbb{R}}}
\def\bbz{{\mathbb{Z}}}

\def\lra{\longrightarrow}

\def\x{\times}

\def\id{\mathrm{id}}
\def\aut{\mathrm{Aut}}

\def\diag{\mathrm{diag}}

\def\R{\mathrm{(R)}}

\def\gfix{\mathrm{fix}}

\def\Sol{\mathrm{Sol}}

\def\tr{\mathrm{tr\,}}

\def\GL{\mathrm{GL}}

\def\bfa{\mathbf{a}}
\def\bfb{\mathbf{b}}

\def\coin{\mathrm{coin}}

\def\calR{\mathcal{R}}

\def\im{\mathrm{im}}
\def\GammaA{\Gamma_{\!A}}
\def\bfm{\mathbf{m}}

\def\bfr{\mathbf{r}}

\def\bfu{\mathbf{u}}
\def\bfM{\mathbf{M}}

\def\boxit#1{\vbox{\hrule\hbox{\vrule\kern3pt
     \vbox{\kern3pt#1\kern3pt}\kern3pt\vrule}\hrule}}

\begin{document}
\title[The $R_\infty$ property for crystallographic groups of $\Sol$]
{The $R_\infty$ property \\for crystallographic groups of $\Sol$}

\author{Ku Yong Ha}
\author{Jong Bum Lee}
\address{Department of Mathematics, Sogang University, Seoul 121-742, KOREA}
\email{kyha@sogang.ac.kr}
\email{jlee@sogang.ac.kr}

\thanks{The second-named author is partially supported by Basic Science Researcher Program through the National Research Foundation of Korea funded by the Ministry of Education (No. 2013R1A1A2058693) and by the Sogang University Research Grant of 2010(10022).}
\thanks{Tel: +82-2-705-8414; Fax: +82-2-714-6284 (J.B.Lee)}


\abstract
There are $9$ kinds of crystallographic groups $\Pi$ of $\Sol$. For any automorphism $\varphi$ on $\Pi$, we study the Reidemeister number $R(\varphi)$. Using the averaging formula for the Reidemeister numbers, we prove that most of the crystallographic groups of $\Sol$ have the $R_\infty$ property.
\endabstract
\maketitle


\section{Introduction}

Let $G_1$ and $G_2$ be groups and $\varphi,\psi:G_1\to G_2$ be group homomorphisms. Then the coincidence group $\coin(\varphi,\psi)$ is define to be
$$
\coin(\varphi,\psi)=\{g \in G_1\mid \varphi(g)=\psi(g)\}.
$$
We also define an equivalence relation $\sim$ on $G_2$ by
$$
\alpha\sim\beta\ \Leftrightarrow\ \beta=\psi(\gamma)\alpha\varphi(\gamma)^{-1}\
\text{ for some $\gamma\in G_1$}.
$$
The equivalence classes are called \emph{Reidemeister coincidence classes} and $\calR[\varphi,\psi]$ denotes the set of Reidemeister coincidence classes.
The \emph{Reidemeister coincidence number} $R(\varphi,\psi)$ of $\varphi,\psi$ is defined to be the cardinality of $\calR[\varphi,\psi]$.

A special case of the Reidemeister coincidence number is the Reidemeister number.

\begin{Def}
Let $G$ be a group and $\varphi:G\to G$ be a group homomorphism. The \emph{Reidemeister number} $R(\varphi)$ of $\varphi$ is defined to be $R(\varphi)=R(\varphi,\id_G)$. We say that $G$ has the \emph{$R_\infty$ property} if $R(\varphi)=\infty$ for every automorphism $\varphi:G\to G$.
\end{Def}

Suppose we have a commutative diagram of groups:
\begin{align*}
\CD
1 @>>> \Gamma_1 @>{i_1}>> \Pi_1 @>{u_1}>> \Pi_1/\Gamma_1 @>>> 1\\
@. @V{\varphi'}V{\psi'}V @V{\varphi}V{\psi}V @V{\bar{\varphi}}V{\bar\psi}V \\
1 @>>> \Gamma_2 @>{i_2}>> {\Pi_2} @>{u_2}>> {\Pi_2}/\Gamma_2 @>>> 1
\endCD
\end{align*}
where the top and bottom sequences are exact and where the quotient groups $\Pi_1/\Gamma_1$ and ${\Pi_2}/\Gamma_2$ are finite.
For each $\bar\alpha\in{\Pi_2}/\Gamma_2$ and $\alpha\in u_2^{-1}(\bar\alpha)$, we have a commutative
diagram
\begin{align*}
\CD
1 @>>> \Gamma_1 @>{i_1}>> \Pi_1 @>{u_1}>> \Pi_1/\Gamma_1 @>>> 1\\
@. @V{\tau_\alpha\varphi'}V{\psi'}V @V{\tau_\alpha\varphi}V{\psi}V
@V{\tau_{\bar\alpha}\bar{\varphi}}V{\bar\psi}V \\
1 @>>> \Gamma_2 @>{i_2}>> {\Pi_2} @>{u_2}>> {\Pi_2}/\Gamma_2 @>>> 1
\endCD
\end{align*}
Here $\tau_\alpha$ is the homomorphism defined by conjugating $\alpha$. Moreover the following sequence of coincidence groups
$$
1\to \coin(\tau_\alpha\varphi', \psi')\stackrel{i_1}{\lra}
\coin(\tau_\alpha\varphi, \psi)\stackrel{u_1}{\lra}
\coin(\tau_{\bar\alpha}\bar\varphi, \bar\psi)
$$
is exact.
Remark that $i_2:\Gamma_2\to\Pi_2$ and $u_2:\Pi_2\to\Pi_2/\Gamma_2$ induce maps $\hat{i}_2^\alpha:\mathcal{R}[\tau_\alpha\varphi',\psi']\to\mathcal{R}[\tau_\alpha\varphi,\psi]$ and $\hat{u}_2^\alpha:\mathcal{R}[\tau_\alpha\varphi,\psi] \to\mathcal{R}[\tau_{\bar{\alpha}}\bar{\varphi},\bar{\psi}]$ such that $\hat{u}_2^\alpha$ is surjective and $(\hat{u}_2^\alpha)^{-1}([\bar{1}])=\mathrm{im}(\hat{i}_2^\alpha)$.
That is, the following sequence of sets is exact:
$$
\mathcal{R}[\tau_\alpha\varphi',\psi']\buildrel{\hat{i}_2^\alpha}\over\lra
\mathcal{R}[\tau_\alpha\varphi,\psi]\buildrel{\hat{u}_2^\alpha}\over\lra
\mathcal{R}[\tau_{\bar{\alpha}}\bar{\varphi},\bar{\psi}]\lra1.
$$

Analyzing the above exact sequences, we obtain the following averaging inequality for Reidemeister numbers:

\begin{theorem}[{\cite[Corollary~3.4, Theorem~3.5]{HLP2}}]\label{HLP}
Suppose we are given the above commutative diagram. Then:
\begin{enumerate}
\item[$(1)$] $R(\varphi,\psi)$ is finite if and only if $R(\tau_\alpha\varphi',\psi')$ is finite for every $\alpha\in\Pi_2$.
\item[$(2)$] We have
$$
R(\varphi,\psi)\ge\frac{1}{[\Pi_1:\Gamma_1]}\sum_{\bar{\alpha}\in\Pi_2/\Gamma_2}R(\tau_\alpha\varphi',\psi').
$$
When either side of the inequality is finite, then equality occurs if and only if $\coin(\tau_{\alpha}\varphi,\psi)\subset\Gamma_1$ for each $\alpha\in\Pi_2$.
\end{enumerate}
\end{theorem}

We have shown in \cite[Theorem~4.2]{HLP2} that the above averaging inequality becomes identity when $\Pi_i$ are orientable Bieberbach groups of simply connected nilpotent Lie groups of equal dimension. We generalize this result to Bieberbach groups of simply connected solvable Lie groups of equal dimension.

\begin{corollary}\label{Reidemeister-Bieb}
Suppose in the above commutative diagram that $\Pi_i$ are {torsion free} extensions of polycyclic groups $\Gamma_i$ by finite groups $\Pi_i/\Gamma_i$. If $\Gamma_1$ and $\Gamma_2$ have the same {Hirsch length}, then \begin{align*}
R(\varphi,\psi)&=\frac{1}{[\Pi_1:\Gamma_1]}\sum_{\bar\alpha\in\Pi_2/\Gamma_2}R(\tau_\alpha\varphi',\psi').
\end{align*}
\end{corollary}

\begin{proof}
By Theorem~\ref{HLP}.(1), we may assume $R(\varphi,\psi)<\infty$. Then $R(\tau_\alpha\varphi',\psi')$ is finite for every $\alpha\in\Pi_2$. According to \cite[Theorem~3.2]{Wong}, $\coin(\tau_\alpha\varphi',\psi')$ is a trivial group. Since $\Pi_1/\Gamma_1$ is a finite group, the subgroup $\coin(\tau_{\bar\alpha}\bar\varphi,\bar\psi)$ of $\Pi_1/\Gamma_1$ is a finite group. From the above exact sequence, $\coin(\tau_\alpha\varphi,\psi)$ is a finite group in the torsion free group $\Pi_1$ and hence it is a trivial group. Now the result follows from Theorem~\ref{HLP}.(2).
\end{proof}

Our aim is to understand the Reidemeister numbers of automorphism on crystallographic groups of simply connected solvable Lie groups. For the $R_\infty$ property of low-dimensional crystallographic groups modeled on simply connected nilpotent Lie groups, we refer to \cite{DP}.
In this paper, we will consider the $3$-dimensional simply connected solvable Lie group $\Sol$, and we shall study the $R_\infty$ property of crystallographic groups of $\Sol$ using Theorem~\ref{HLP} and Corollary~\ref{Reidemeister-Bieb}. The work of discovering which groups have the $R_\infty$ property was begun by Fel'shtyn and Hill in \cite{FH}.
\bigskip

\section{The crystallographic groups of $\Sol$}

One can describe $\Sol$ as a semi-direct product $\bbr^2\rtimes_\sigma\bbr$ where
$t\in\bbr$ acts on $\bbr^2$ via the matrix
$$
\sigma(t)=\left[\begin{matrix}e^t&0\\0&e^{-t}\end{matrix}\right].
$$
The group of affine automorphisms of $\Sol$ is $\mathrm{Aff}(\Sol)=\Sol\rtimes\aut(\Sol)$. Let $K$ be a maximal compact subgroup of $\aut(\Sol)$. A discrete cocompact subgroup of $\Sol\rtimes K\subset \mathrm{Aff}(\Sol)$ is called a \emph{crystallographic group} modeled on $\Sol$, simply an SC-\emph{group} of $\Sol$. A torsion free SC-group is called a \emph{Bieberbach group} or an SB-\emph{group} of $\Sol$.

Let $\Pi$ be an SC-group of $\Sol$. Let $\Gamma=\Pi\cap\Sol$ and $\Phi=\Pi/\Gamma$. Then $\Gamma$ is a lattice (i.e., a discrete cocompact subgroup) of $\Sol$ and $\Phi$ is a finite group, called the \emph{holonomy group} of $\Pi$.

Now we recall from \cite{LZ} that a lattice of $\Sol$ is determined by a $2\x2$ hyperbolic integer matrix $A$ of determinant $1$ and trace $>2$. Let $\Gamma$ be a lattice of $\Sol$. Then $\bbr^2\cap\Gamma$ is a lattice of $\bbr^2$ and $\Gamma/\bbr^2\cap\Gamma$ is a lattice of $\Sol/\bbr^2=\bbr$, so that $\bbr^2\cap\Gamma\cong\bbz^2$ and
$\Gamma/\bbr^2\cap\Gamma\cong\bbz$, and the following diagram of short exact sequences is commutative
$$
\CD 1@>>>\bbr^2@>>>\Sol@>>>\bbr@>>>1\\
@.@AAA@AAA@AAA\\
1@>>>\bbz^2@>>>\Gamma@>>>\bbz@>>>1
\endCD
$$
Choose a basis $\{\bfx_1,\bfx_2\}$ for $\bbz^2$ and a basis $t_0$
for $\bbz$. Then
\begin{align*}
\sigma(t_0)(\bfx_i)= \ell_{1i}\bfx_1+\ell_{2i}\bfx_2,\ (i=1,2)
\end{align*}
for some integers $\ell_{ij}$. Thus the lattice $\Gamma$ is a subgroup of $\Sol$ generated by $\bfx_1,\bfx_2$ and $t_0$ satisfying the above identity.
Let $P=\big[\bfx_1\ \bfx_2\big]$ be the matrix with columns $\bfx_1$ and $\bfx_2$, and let
$$
A=\left[\begin{matrix} \ell_{11}&\ell_{12}\\\ell_{21}&\ell_{22}
\end{matrix}\right].
$$
Then
$$
PAP^{-1}=\sigma(t_0)=\left[\begin{matrix}e^{t_0}&0\\0&e^{-t_0}\end{matrix}\right]
$$
and so $A\in \mathrm{SL}(2,\bbz)$. Note that $P^{-1}$ consists of eigenvectors of $A$ with eigenvalues $e^{t_0}$ and $e^{-t_0}$. Notice also that $A$ has trace $e^{t_0}+e^{-t_0}=\ell_{11}+\ell_{22}>2$. This implies that $A$ is a hyperbolic matrix; it has different real eigenvalues: one is greater than $1$ and the other is less than $1$. Furthermore, neither $\ell_{12}$ nor $\ell_{21}$ vanishes, see for example \cite{LZ}. We shall denote such a lattice by $\GammaA$. Then
\begin{align*}
\GammaA&=\langle{a_1,a_2,t\ |\ [a_1,a_2]=1, ta_it^{-1}=a_1^{\ell_{1i}}a_2^{\ell_{2i}}}\rangle\\
&=\langle{a_1,a_2,t\ |\ [a_1,a_2]=1, ta_it^{-1}=A(a_i)}\rangle.
\end{align*}
For an element of the form $a_1^xa_2^y$, we shall use the notation $\bfa^\bfx$.

It is known from \cite[Theorem~8.2]{HL2} that there are $9$ kinds of SC-groups of $\Sol$:
{$\GammaA$, $\Pi_1(\bfk)$, $\Pi_2^\pm$, $\Pi_3(\bfk,\bfk')$, $\Pi_4(\bfk)$, $\Pi_5(\bfm,\bfk,\bfk',\bfn)$, $\Pi_6(\bfk,\bfk')$, $\Pi_7(\bfk)$ and $\Pi_8(\bfk,\bfm)$}.
There are $4$ kinds of SB-groups of $\Sol$.
We recall from \cite[Corollary~8.3]{HL2} that $\GammaA$ and $\Pi_2^\pm$ are SB-groups, and the SC-groups $\Pi_1(\bfk)$, $\Pi_4(\bfk)$, $\Pi_5(\bfm,\bfk,\bfk',\bfn)$, $\Pi_7(\bfk)$ and $\Pi_8(\bfk,\bfm)$ are not SB-groups. The SC-groups $\Pi_3(\bfk,\bfk')$ and $\Pi_6(\bfk,\bfk')$ become SB-groups for a particular choice of $\bfk$ and $\bfk'$. In fact, we may assume
$$
M=\left[\begin{matrix}-1&m\\\hspace{8pt}0&1\end{matrix}\right]
$$
where $m=0$ or $1$. If $m=0$, then $\ell_{11}=\ell_{22}$ and $\ker(I-M)/\im(I+M)\cong\bbz_2$ is generated by $\bfe_2=(0,1)^t$. If $m=1$, then $\ell_{11}-\ell_{22}=\ell_{21}$ and $\ker(I-M)/\im(I+M)$ is a trivial group and hence $\bfk={\bf0}$. It is shown in \cite[Corollary~8.3]{HL3} that they are SB-groups if and only if $m=0$, $\bfk=\bfe_2$ and $\bfk'-\bfk\ne{\bf0}$. Thus they are not SB-groups if and only if
\begin{enumerate}
\item $m=1$,
\item $m=0$ and $\bfk={\bf0}$, or
\item $m=0$ and $\bfk=\bfk'=\bfe_2$.
\end{enumerate}

\begin{rmk}\label{FI}
Let $\varphi:\Pi\to\Pi$ be an automorphism on an SC-group $\Pi$. By \cite[Lemma~2.1]{LL2}, there is a fully invariant subgroup $\Lambda\subset\GammaA=\Pi\cap\Sol$ of $\Pi$, which is of finite index. Since $\GammaA$ is generated by $a_1,a_2,t$, it follows that $\Lambda$ is generated by some elements $b_1=\bfa^{\bfm_1}, b_2=\bfa^{\bfm_2}, s=\bfa^* t^k$. Furthermore, the subgroup $\langle b_1, b_2 \rangle$ is a fully invariant subgroup of the lattice $\Lambda$ (see for example \cite[Theorem~2.3]{LZ}).
\end{rmk}

Let $\varphi$ be a homomorphism on $\Pi$, and let $\Lambda=\langle b_1,b_2,s\rangle\subset\Pi$ be a lattice of $\Sol$ such that $\varphi(\Lambda)\subset\Lambda$. Denote by $\varphi'$ the homomorphism obtained by restricting $\varphi$ on $\Lambda$. Then by \cite[Theorem~2.4]{LZ}, $\varphi(b_i)=\bfb^{\bfn_i}$ and $\varphi(s)=\bfb^\bfp s^m$ for some $\bfn_i,\bfp\in\bbz^2$ and $m\in\bbz$. We say that $\varphi$ or $\varphi'$ is of type (I) if $m=1$; of type (II) if $m=-1$; of type (III) if $m\ne\pm1$. When $\varphi$ is of type (III), we have $\varphi(b_i)=1$.

In the following sections we will show that the SC-groups $\Pi_2^-$, $\Pi_3(\bfk,\bfk')$, $\Pi_4(\bfk)$, $\Pi_5(\bfm,\bfk,\bfk',\bfn)$, $\Pi_6(\bfk,\bfk')$, $\Pi_7(\bfk)$ and $\Pi_8(\bfk,\bfm)$ have the $R_\infty$ property using Theorem~\ref{HLP} and Theorem~\ref{Pi_1-R}. In order to apply Theorem~\ref{HLP}, we need to find a characteristic subgroup $\subset\GammaA$ of each SC-group. It turns out that $\GammaA$ itself is a characteristic subgroup, not a fully invariant subgroup, of all the SC-groups except $\Pi_5(\bfm,\bfk,\bfk',\bfn)$.

On the other hands, the groups $\GammaA$, $\Pi_1(\bfk)$ and $\Pi_2^+$ have finite Reidemeister numbers of automorphisms $\varphi$ only when $\varphi$ is of type (II) with $\det\varphi'=-1$.
When the automorphism $\varphi$ is of type (II), it is obvious that $\varphi^2$ is of type (I). Therefore, $R(\varphi^2)=\infty$. In particular, the Reidemeister zeta function \cite{Fel00}
$$
R_\varphi(z)=\exp\left(\sum_{n=1}^\infty\frac{R(\varphi^n)}{n}z^n\right)
$$
is not defined for any automorphism $\varphi$ of an SC-group of $\Sol$. In fact, it is shown in \cite{FL} that if the Reidemeister zeta function is defined for an automorphism on an infra-solvmanifold of type $\R$, then the manifold is an infra-nilmanifold.
\bigskip

\section{The SB-groups}

In this section, we will study the Reidemeister numbers of automorphisms on SB-groups of $\Sol$ (see also Section~4 of \cite{GW}, in which a different method is used).

\begin{theorem}\label{Pi_1-R}
For any automorphism $\varphi$ on $\GammaA$, we have
$$
R(\varphi)=\begin{cases}
4&\text{when $\varphi$ is of type $(\mathrm{II})$ and $\det\varphi=-1;$}\\
\infty&\text{otherwise.}
\end{cases}
$$
\end{theorem}

\begin{proof}
Let $\varphi:\GammaA\to\GammaA$ be an automorphism. Then it is determined by
\begin{align*}
\varphi(a_1)=\bfa^{\bfu_{1}},\
\varphi(a_2)=\bfa^{\bfu_{2}},\
\varphi(\tau)=\bfa^{\bfp}\tau^{\pm1}
\end{align*}
where $\det[\bfu_1\ \bfu_{2}]=\det[\varphi]=\pm1$.
Notice that every element of $\GammaA$ is of the form $\bfa^\bfx \tau^z$. Its Reidemeister class is
$$
[\bfa^\bfx \tau^z] =\left\{(\bfa^\bfq\tau^m)(\bfa^\bfx \tau^z)\varphi(\bfa^\bfq\tau^m)^{-1}\mid \bfq\in\bbz^2,m\in\bbz \right\}.
$$
A simple computation shows that
$$
(\bfa^{\bfq}\tau^m)(\bfa^\bfx \tau^z)\varphi(\bfa^\bfq\tau^m)^{-1}=\bfa^*\tau^{z+m\mp m}.
$$
This implies that when $\varphi$ is of type (I), the distinct $z$'s yield distinct Reidemeister classes and so $R(\varphi)=\infty$.

Assume $\varphi$ is of type (II). Since
\begin{align*}
\varphi(\tau)^{m}&=(\bfa^\bfp\tau^{-1})^{m}\\
&=\begin{cases}
A^{-m}(A+A^2+\cdots+A^m)(\bfa^\bfp)\tau^{-m}, &\text{$m>0$;}\\
(A+A^2+\cdots+A^{-m})(\bfa^{-\bfp})\tau^{-m}, &\text{$m<0$,}
\end{cases}
\end{align*}
we have
\begin{align*}
&(\bfa^\bfq\tau^m)(\bfa^\bfx\tau^z)\varphi(\bfa^\bfq\tau^m)^{-1}=\bfa^\bfM\tau^{z+2m}
\end{align*}
where
\begin{align*}
\bfM
&=\left(I-A^{z+2m}[\varphi]\right)\bfq+A^m\bfx+\bfp_m,\\
\bfp_m&=
\begin{cases}
-\hspace{1pt}A^{z+m}\hspace{1pt}(A+A^2+\cdots+A^m)\bfp, &\text{$m>0$;}\\
A^{z+2m}(A+A^2+\cdots+A^{-m})\bfp, &\text{$m<0$;}\\
{\bf0}, &\text{$m=0$.}
\end{cases}
\end{align*}
We recall that there is $P$ such that
$$
PAP^{-1}=\left[\begin{matrix}e^{t_0}&0\\0&e^{-t_0}\end{matrix}\right],\
P[\varphi]P^{-1}=\left[\begin{matrix}0&\gamma\\\delta&0\end{matrix}\right].
$$
It follows that $\det(I-A^{z+2m}[\varphi])=1+\det[\varphi]$. If $z=0$, then
\begin{align*}
\bfM=\left(I-A^{2m}[\varphi]\right)\bfq+A^m\bfx+\bfp_m.
\end{align*}
If $\det\varphi=1$, then $\det[\varphi]=-1$ and $\det\left(I-A^{2m}[\varphi]\right)=0$ and so there are infinitely many Reidemeister classes and $R(\varphi)=\infty$. Assume $\det\varphi=-1$ or $\det[\varphi]=1$. Then $\det\left(I-A^{2m}[\varphi]\right)=2$ for all $m$. From this, we can show that there are four Reidemeister classes: $\{[1], [\tau], [\bfa^{\bfx_0}], [\bfa^{\bfx_1}\tau]\}$ where $\bfx_0\notin\im(I-[\varphi])$ and $\bfx_1\notin\im(I-A[\varphi])$. So, $R(\varphi)=4$.
\end{proof}

Recall that
$$
\Pi_2^\pm=\langle{a_1, a_2, \beta \ |
\begin{array}{l}
[a_1,a_2]=1,\ \beta a_i\beta^{-1}=N_\pm(a_i)
\end{array}}\rangle,
$$
where $N_\pm$ are square roots of $A$:
$$
N_\pm=-\left[\begin{matrix}\frac{\ell_{11}\pm1}{\sqrt{\ell_{11}+\ell_{22}\pm2}}
&\frac{\ell_{12}}{\sqrt{\ell_{11}+\ell_{22}\pm2}}\\
\frac{\ell_{21}}{\sqrt{\ell_{11}+\ell_{22}\pm2}}
&\frac{\ell_{22}\pm1}{\sqrt{\ell_{11}+\ell_{22}\pm2}}
\end{matrix}\right].
$$
Let $\varphi:\Pi_2^\pm\to\Pi_2^\pm$ be an automorphism. Then $\GammaA$ is a fully invariant subgroup of $\Pi_2^\pm$ and thus we have the following commutative diagram:
$$
\CD
1@>>>\GammaA@>>>\Pi_2^\pm@>>>\Phi_2^\pm@>>>1\\
@.@VV{\varphi'}V@VV{\varphi}V@VV{\bar\varphi}V\\
1@>>>\GammaA@>>>\Pi_2^\pm@>>>\Phi_2^\pm@>>>1
\endCD
$$
By Corollary~\ref{Reidemeister-Bieb}, we have
$$
R(\varphi)= \frac{1}{2}\left(R(\varphi')+R(\tau_{\sigma}\varphi')\right).
$$
Notice that $[\tau_\sigma\varphi']=N_\pm[\varphi']$. Thus $\varphi'$ and $\tau_\sigma\varphi'$ have the same type, and $\det\tau_\sigma\varphi'=\det\varphi'$ for $\Pi_2^+$ and $\det\tau_\sigma\varphi'=-\det\varphi'$ for $\Pi_2^-$. Assume $\varphi'$ is of type (II) with $\det\varphi'=-1$.
Recalling from Section~5.2 of \cite{HL3} that any automorphism on $\Pi_2^-$ cannot be of type (II), $\varphi'$ and $\tau_\sigma\varphi'$ are on $\Pi_2^+$ of type (II), and $\det\tau_\sigma\varphi'=-1$. Hence we have:

\begin{theorem}\label{GA}
For any automorphism $\varphi$ on $\Pi_2^+$, we have
$$
R(\varphi)=\begin{cases}
4&\text{when $\varphi$ is of type $\mathrm{(II)}$ and $\det\varphi'=-1$;}\\
\infty&\text{otherwise.}
\end{cases}
$$
The $\mathrm{SB}$-group $\Pi_2^-$ has the $R_\infty$ property.
\end{theorem}

Next we consider the SB-groups $\Pi_3(\bfk,\bfk')$ and $\Pi_6(\bfk,\bfk')$.
\begin{theorem}\label{Pi_34-R}
The SB-groups $\Pi_3(\bfk,\bfk')$ and $\Pi_6(\bfk,\bfk')$ have the $R_\infty$ property.
\end{theorem}

\begin{proof}
Denote the SB-groups $\Pi_3(\bfk,\bfk')$ and $\Pi_6(\bfk,\bfk')$ by $\Pi_3$ and $\Pi_6$ respectively.
Let $\varphi:\Pi_3\to\Pi_3$ be an automorphism. Since $\GammaA$ is a fully invariant subgroup of $\Pi_3$ (cf. \cite[Lemma~5.3]{HL3}), we have the following commutative diagram:
$$
\CD
1@>>>\GammaA@>>>\Pi_3@>>>\Phi_3@>>>1\\
@.@VV{\varphi'}V@VV{\varphi}V@VV{\bar\varphi}V\\
1@>>>\GammaA@>>>\Pi_3@>>>\Phi_3@>>>1
\endCD
$$
By Corollary~\ref{Reidemeister-Bieb}, we have
$$
R(\varphi)= \frac{1}{2}\left(R(\varphi')+R(\tau_{\alpha}\varphi')\right).
$$
Assume $\varphi'$ is of type (II) with $\det\varphi'=-1$. Then $R(\varphi')=4$. However, since $\tau_\alpha\varphi'(\tau)=\alpha\varphi(\tau)\alpha^{-1}=\alpha(\bfa^*\tau^{-1})\alpha^{-1}=\bfa^*\tau$, it follows that $\tau_\alpha\varphi'$ is of type (I) and hence $R(\tau_\alpha\varphi')=\infty$. In all, $R(\varphi)=\infty$.

Let $\varphi:\Pi_6\to\Pi_6$ be an automorphism. Then we have the following commutative diagram:
$$
\CD
1@>>>\GammaA@>>>\Pi_6@>>>\Phi_4@>>>1\\
@.@VV{\varphi'}V@VV{\varphi}V@VV{\bar\varphi}V\\
1@>>>\GammaA@>>>\Pi_6@>>>\Phi_4@>>>1
\endCD
$$
By Corollary~\ref{Reidemeister-Bieb}, we have
$$
R(\varphi)= \frac{1}{4}\left(R(\varphi')+R(\tau_{\sigma}\varphi')+R(\tau_{\alpha}\varphi')+R(\tau_{\sigma\alpha}\varphi')\right).
$$
Assume $\varphi'$ is of type (II) with $\det\varphi'=-1$. Since $\tau_\alpha\varphi'(\sigma^2)=\alpha\varphi(\sigma^2)\alpha^{-1}=\alpha(\bfa^*\sigma^{-2})\alpha^{-1}=\bfa^*\sigma^2$, it follows that $\tau_\alpha\varphi'$ is of type (I) and hence $R(\tau_\alpha\varphi')=\infty$. In all, $R(\varphi)=\infty$.
\end{proof}

Our aim is to continue the study of the $R_\infty$ property for the remaining SC-groups. In the following sections, we shall find a maximal characteristic subgroup of every SC-group $\Pi$ using Remark~\ref{FI} and then we use Theorem~\ref{HLP} to compute the Reidemeister numbers of all automorphisms on $\Pi$.
\bigskip

\section{The SC-groups $\Pi_1(\bfk)$}

Recall that
$$
\Pi_1(\bfk)=\left\langle{a_1, a_2, t, \beta \ \Big|
\begin{array}{l}
[a_1,a_2]=1,ta_it^{-1}=A(a_i),\\
\beta a_i\beta^{-1}=\bfa^{-\bfe_i},\ \beta^2=1,\
\beta t\beta^{-1}=\bfa^{\bfk}t
\end{array}}\right\rangle,
$$
where
$$
\bfe_1=(1,0)^t,\ \bfe_2=(0,1)^t,\ \bfk\in \frac{\bbz^2}{\left(2(\bbz^2)+\im(I-A)\right)}.
$$

\begin{lemma}\label{ch1}
Let $\varphi:\Pi_1(\bfk)\to\Pi_1(\bfk)$ be an automorphism. Then
\begin{align*}
&\varphi(a_i)=\bfa^{\bfn_i}, \,\, \varphi(t)=\bfa^{\bfp}t^{\pm1},\,\,
\varphi(\beta)=\bfa^{\bfx}\beta
\end{align*}
for some $\bfn_i,\bfp,\bfx\in\bbz^2$ satisfying the following conditions
\begin{align*}
&[\varphi]=[\bfn_1\ \bfn_2]=[n_{ij}]\in\GL(2,\bbz),\\
&A^{\pm1}[\varphi]=[\varphi] A,\\
&2\bfp=\begin{cases}
(I-A)\bfx+(I-[\varphi])\bfk &\text{when } \varphi(t)=\bfa^{\bfp}t;\\
(I-A^{-1})\bfx-(A^{-1}+[\varphi])\bfk &\text{when } \varphi(t)=\bfa^{\bfp}t^{-1}.
\end{cases}
\end{align*}
In particular, the subgroup $\GammaA=\langle a_1,a_2,t\rangle$ of $\Pi_1(\bfk)$ is a characteristic subgroup.
\end{lemma}

\begin{proof}
Every element of $\Pi_1(\bfk)$ is of the form $\bfa^\bfx t^z\beta^w$ with $w\in\{0,1\}$.
Suppose $\varphi:\Pi_1(\bfk)\to\Pi_1(\bfk)$ is an automorphism.
Since $\beta$ is a torsion element of order $2$, so is $\varphi(\beta)$. It follows that $\varphi(\beta)=\bfa^{\bfx}\beta$.

If $\varphi(a_i)=\bfa^{\bfn_i} t^{m_i} \beta^w$,
then $\beta a_i\beta^{-1}=a_i^{-1} \Rightarrow m_i=0$.
Since $\bfa^{\bfn_i}\beta$ is torsion of order $2$ and $\varphi$ is an automorphism,
$\varphi(a_i)=\bfa^{\bfn_i}$.

We have shown that the subgroup $\langle a_1,a_2,\beta\rangle\subset\Pi_1(\bfk)$ is characteristic and hence $\varphi$ induces an automorphism on the quotient group $\Pi_1(\bfk)/\langle a_1,a_2,\beta\rangle\cong\bbz$. This implies that $\varphi(t)=\bfa^{\bfp}t^{\pm1}\beta^w$.
Assume $\varphi(t)=\bfa^{\bfp}t^{\pm1}\beta$. Then
$$
ta_it^{-1}=A(a_i) \Rightarrow -A^{\pm1}[\varphi]=[\varphi] A \quad\text{where }\,[\varphi]=[\bfn_1\ \bfn_2]=[n_{ij}].
$$
We choose an invertible matrix $P$ so that $PAP^{-1}=\diag\{e^{t_0},e^{-t_0}\}=D$ (see \cite[Remark~5.5]{HL2}). Let $Q=P[\varphi]P^{-1}$. Then $-D^{\pm1}Q=QD$. This induces $Q=[\varphi]=0$, contradicting that $\varphi$ is an automorphism. Thus $\varphi(t)=\bfa^{\bfp}t^{\pm1}$. In all, we have shown that $\GammaA$ is a characteristic subgroup of $\Pi_1(\bfk)$.

Observe further that
$ta_it^{-1}=A(a_i)$ induces
\begin{align*}
\varphi(ta_it^{-1})
&=\varphi(A(a_i)) \Rightarrow A^{\pm1}[\varphi]=[\varphi]A.
\end{align*}
From $\beta t \beta^{-1}= \bfa^{\bfk} t$, we also have
\begin{align*}
\varphi(\beta t \beta^{-1})
&=\varphi(\bfa^{\bfk})\varphi(t)
\Rightarrow (\bfa^{\bfx}\beta) (\bfa^{\bfp}t^{\pm1}) (\bfa^{\bfx}\beta)^{-1}=\varphi(\bfa^{\bfk})\bfa^{\bfp}t^{\pm1}.
\end{align*}
This identity induces
\begin{align*}
2\bfp&=\begin{cases}
(I-A)\bfx+(I-[\varphi])\bfk &\text{when } \varphi(t)=\bfa^{\bfp}t;\\
(I-A^{-1})\bfx-(A^{-1}+[\varphi])\bfk &\text{when } \varphi(t)=\bfa^{\bfp}t^{-1}.
\end{cases}
\end{align*}
This finishes the proof.
\end{proof}

\begin{rmk}
Consider a homomorphism $\varphi:\Pi_1(\bfk)\to\Pi_1(\bfk)$ defined by
$$
\varphi(\bfa^\bfx t^z\beta^w)=\beta^z, w=0,1.
$$
It is clear that $\varphi$ is not an automorphism and $\varphi(\GammaA)\nsubseteq\GammaA$. Thus $\GammaA$ is not a fully invariant subgroup of $\Pi_1(\bfk)$.
\end{rmk}

\begin{theorem} \label{Reid1}
Let $\varphi:\Pi_1(\bfk)\to\Pi_1(\bfk)$ be an automorphism and let $\varphi'=\varphi|_{\GammaA}$. Then $R(\varphi)=\infty$ if and only if $\varphi'$ is of type $\mathrm{(I)}$ or type $\mathrm{(III)}$ or type $\mathrm{(II)}$ with $\det\varphi'=1$. If $\varphi'$ is of type $\mathrm{(II)}$ with $\det\varphi'=-1$, then $4\le R(\varphi)<\infty$.
\end{theorem}

\begin{proof}
By Theorem~\ref{HLP}, we have
$$
R(\varphi)\ge\frac{1}{2}\left(R(\varphi')+R(\tau_\beta\varphi')\right).
$$
Observe further that
\begin{align*}
&\varphi'(a_i) =\bfa^{\bfn_i}, \,
&&\varphi'(t)=\bfa^{*}t^{\omega} \,\,(\omega\in\{\pm 1\}), \\
&\tau_\beta\varphi'(a_i)=\bfa^{-\bfn_i},\,
&&\tau_\beta\varphi'(t)=\bfa^{*}t^{\pm\omega}.
\end{align*}
Hence $\varphi'$ and $\tau_\beta\varphi'$ have the same type and $\det\varphi'=\det\tau_\beta\varphi'$. Thus $R(\varphi)=\infty$ if and only if $R(\varphi')=\infty$ if and only if $\varphi'$ is of type (I) or type (III) or type (II) with $\det\varphi'=1$ by Theorem~\ref{GA}. This proves the theorem.
\end{proof}

In the following, we will evaluate the Reidemeister numbers $R(\varphi)$ for all automorphisms $\varphi$ on $\Pi_1(\bfk)$ of type (II) with $\det\varphi'=-1$. This is exactly the case
when $R(\varphi)<\infty$, and $R(\varphi')=4$.

In this case, the corresponding $\tau_\beta\varphi'$ is also of type (II) and $\det\tau_\beta\varphi'=-1$, and so
$R(\tau_\beta\varphi')=4$.
Hence, from Theorem~\ref{HLP}, we have $4\le R(\varphi)<\infty$ and equality occurs if and only if $\gfix(\tau_\alpha\varphi)\subset\GammaA$ for all $\alpha\in\Pi_1(\bfk)$. Furthermore, since $\det\varphi'=-1$, the conditions of Lemma~\ref{ch1} become
\begin{align*}
&[\varphi]=\left[\begin{matrix}
-u&\frac{(\ell_{11}-\ell_{22})u-\ell_{12}v}{\ell_{21}}
\\\hspace{8pt}v&u\end{matrix}\right]
:=\left[\begin{matrix}-u&v'\\\hspace{8pt}v&u
\end{matrix}\right]\in\mathrm{SL}(2,\bbz),\\
&2\bfp=(I-A^{-1})\bfx-(A^{-1}+[\varphi])\bfk.
\end{align*}

We will find conditions on $\varphi$ for which both $\gfix(\varphi),\gfix(\tau_\beta\varphi)\subset\GammaA$. By Lemma~\ref{ch1}, we have $\varphi(t)=\bfa^\bfp t^{-1}$ and so $\tau_\beta\varphi(t)=\bfa^*t^{-1}$.
Let $\bfa^\bfr t^m\beta^w$ be fixed by $\varphi$ or $\tau_\beta\varphi$. Then we can see easily that $m=0$ and hence
$\gfix(\varphi),\gfix(\tau_\beta\varphi)\subset\{\bfa^{\bfr}\beta^w\mid \bfr\in\bbz^2, w\in\{0,1\}\}$.
Since $\varphi'$ is of type (II) and $\det\varphi'=-1$, we have $\det[\varphi]=1$ and so $\det(I\pm[\varphi])=1+\det[\varphi]=2$ and $(I-[\varphi])(I+[\varphi])=2I$. Thus, we have
\begin{align*}
\bfa^\bfr\beta\in\gfix(\varphi)
&\Leftrightarrow (I-[\varphi])\bfr=\bfx
\Leftrightarrow \bfr=\frac{1}{2}(I+[\varphi])\bfx\\
\bfa^\bfr\beta\in\gfix(\tau_\beta\varphi)
&\Leftrightarrow (I+[\varphi])\bfr=-\bfx
\Leftrightarrow
\bfr=-\frac{1}{2}(I-[\varphi])\bfx.
\end{align*}
Note that $\frac{1}{2}(I+[\varphi])\bfx +\frac{1}{2}(I-[\varphi])\bfx =\bfx$.
Hence we obtain that
\begin{align*}
\gfix(\varphi)\subset\GammaA &\Leftrightarrow \gfix(\tau_\beta\varphi)\subset\GammaA\\
&\Leftrightarrow (I+[\varphi])\bfx\notin2(\bbz^2) \Leftrightarrow \bfx\notin\im(I-[\varphi])\\
&\Leftrightarrow (I-[\varphi])\bfx\notin2(\bbz^2) \Leftrightarrow \bfx\notin\im(I+[\varphi]).
\end{align*}
Consequently, $R(\varphi)>4$ if and only if $(I\pm[\varphi])\bfx\in2(\bbz^2)$.

\begin{theorem}\label{Reid2}
Let $\varphi:\Pi_1(\bfk)\to\Pi_1(\bfk)$ be an automorphism. If $4< R(\varphi)<\infty$, then $R(\varphi)=8$.
\end{theorem}

\begin{proof}
Since $R(\varphi)$ is finite, $\varphi'=\varphi|_{\GammaA}$ is of type (II) and $\det\varphi'=-1$. The assumption that $R(\varphi)>4$ is equivalent to the condition that $(I\pm[\varphi])\bfx\in2(\bbz^2)$.
In this case, we can see that
\begin{equation}\label{fix-im}
\bfx\in\im(I\pm[\varphi]),\quad A\bfx\pm\bfk \in \im(I\pm A[\varphi])
\end{equation}
because $(I\mp[\varphi])\bfx \in 2(\bbz^2)$ and
\begin{align}
(I\mp[\varphi])(I\pm[\varphi])&=2I=(I\mp A[\varphi])(I\pm A[\varphi]),
\label{inv-im}\\
(I-A[\varphi])(A\bfx-\bfk)
&=2A\bfx-2A\bfp-2\bfk-(I+[\varphi])\bfx \in 2(\bbz^2), \notag\\
(I+A[\varphi])(A\bfx+\bfk)
&=2A\bfx-2A\bfp-(I-[\varphi])\bfx \in 2(\bbz^2)  \notag
\end{align}

Consider any element $\bfa^\bfq t^z\beta^w\in\Pi_1(\bfk)$.
For any $\bfy\in\bbz^2$,
we define
$$
\bfy_m=\begin{cases}
(I+A+\cdots+A^{m-1})\bfy,&\text{$m>0$;}\\
-(A^{-1}+A^{-2}+\cdots+A^{m})\bfy,&\text{$m<0$;}\\
{\bf0},&\text{$m=0$.}
\end{cases}
$$
Then
$$
(\bfa^{\bfy}t)^m =\bfa^{\bfy_m} t^m, \quad
(\bfa^{\bfy}t^{-1})^m =\bfa^{A^{1-m}\bfy_m} t^{-m}.
$$

We will determine its Reidemeister class $[\bfa^\bfq t^z\beta^w]$ as a subset of $\Pi_1(\bfk)$. First we observe that
\begin{align*}
&(\bfa^\bfr t^m)(\bfa^\bfq t^z)\varphi(\bfa^\bfr t^m)^{-1}\\
&\qquad=\bfa^{(I-A^{2m+z}[\varphi])\bfr+A^m\bfq-A^{m+z+1}\bfp_m}t^{2m+z},\\
&(\bfa^\bfr t^m \beta)(\bfa^\bfq t^z)\varphi(\bfa^\bfr t^m \beta)^{-1}\\
&\qquad=\bfa^{(I-A^{2m+z}[\varphi])\bfr-A^m\bfq+A^m\bfk_z-A^{m+z}\bfx-A^{m+z+1}\bfp_m}t^{2m+z},\\
&(\bfa^\bfr t^m)(\bfa^\bfq t^z\beta)\varphi(\bfa^\bfr t^m)^{-1}\\
&\qquad=\bfa^{(I+A^{2m+z}[\varphi])\bfr+A^m\bfq+A^{m+z+1}\bfp_m+A^{m+z}\bfk_m}t^{2m+z}\beta,\\
&(\bfa^\bfr t^m \beta)(\bfa^\bfq t^z\beta)\varphi(\bfa^\bfr t^m \beta)^{-1}\\
&\qquad=\bfa^{(I+A^{2m+z}[\varphi])\bfr-A^m\bfq+A^m\bfk_z+A^{m+z}\bfx+A^{m+z+1}\bfp_m+A^{m+z}\bfk_m} t^{2m+z}\beta.
\end{align*}
Thus every Reidemeister class is one of the following forms $[\bfa^\bfq], [\bfa^{\bfq'}t], [\bfa^\bfq\beta]$ and $[\bfa^{\bfq'}t\beta]$; these classes are distinct each other.

Since $\bfx \in \im(I-[\varphi])$ and $(I-[\varphi])(I+[\varphi])=2I$,
we have
\begin{align*}
[\bfa^{\bfq}]=[\bfa^{\bfq'}]  &\Leftrightarrow
\bfq'-\bfq \in \im(I-[\varphi]) \text{ or }
\bfq'+\bfq +\bfx \in \im(I-[\varphi]) \\
&\Leftrightarrow \bfq' \pm \bfq \in \im(I-[\varphi]) \\
&\Leftrightarrow \bfq' - \bfq \in \im(I-[\varphi]).
\end{align*}
This shows that there are exactly $|\det(I-[\varphi])|=2$ Reidemeister classes of the form $[\bfa^\bfq]$.

Similarly, using the condition~(\ref{fix-im}) and the identity~(\ref{inv-im}), we have
\begin{align*}
&[\bfa^{\bfq}z]=[\bfa^{\bfq'}z] \\
&\qquad\Leftrightarrow
\bfq'-\bfq \in \im(I-A[\varphi]) \text{ or }
\bfq'+\bfq +A\bfx-\bfk \in \im(I-A[\varphi]) \\
&\qquad\Leftrightarrow \bfq' - \bfq \in \im(I-A[\varphi]),\\
&[\bfa^{\bfq}\beta]=[\bfa^{\bfq'}\beta]  \\
&\qquad\Leftrightarrow
\bfq'-\bfq \in \im(I+[\varphi]) \text{ or }
\bfq'+\bfq -\bfx \in \im(I+[\varphi]) \\
&\qquad\Leftrightarrow \bfq' - \bfq \in \im(I+[\varphi]),\\
&[\bfa^{\bfq}z\beta]=[\bfa^{\bfq'}z\beta]  \\
&\qquad\Leftrightarrow
\bfq'-\bfq \in \im(I+A[\varphi]) \text{ or }
\bfq'+\bfq -(A\bfx+\bfk) \in \im(I+A[\varphi]) \\
&\qquad\Leftrightarrow \bfq' - \bfq \in \im(I+A[\varphi]).
\end{align*}
Since $|\det(I\pm A[\varphi])|=|\det(I\pm[\varphi])|=2$,
there are exactly $2$ Reidemeister classes of each of
the forms $[\bfa^\bfq z]$,
$[\bfa^\bfq\beta]$ and $[\bfa^\bfq z\beta]$.

In all, we have shown that there are exactly $2\x4=8$ Reidemeister classes in $\Pi_1(\bfk)$. That is, $R(\varphi)=8$.
\end{proof}

We have shown in Theorems~\ref{Reid1} and \ref{Reid2} that if $R(\varphi)<\infty$ then $\varphi'$ is of type (II) and $\det\varphi'=-1$. In this case, $R(\varphi)=4$ or $8$.

In the SC-groups $\Pi_1(\bfk)$, we recall that $\bfk\in\bbz^2/L$, where the lattice $L:=2(\bbz^2)+\im(I-A)$ is generated by the vectors
$$
\left[\begin{matrix}2\\0\end{matrix}\right],
\left[\begin{matrix}0\\2\end{matrix}\right],
\left[\begin{matrix}1-\ell_{11}\\-\ell_{21}\end{matrix}\right],
\left[\begin{matrix}-\ell_{12}\\1-\ell_{22}\end{matrix}\right].
$$
Thus $\bfk$ depends on $A$.
Indeed we can see that
\begin{align*}
(\ell_{11},\ell_{12},\ell_{21},\ell_{22})
=(e,o,o,o) &\Rightarrow \bfk={\bf0};\\
(\ell_{11},\ell_{12},\ell_{21},\ell_{22})
=(o,o,o,e) &\Rightarrow \bfk={\bf0};\\
(\ell_{11},\ell_{12},\ell_{21},\ell_{22})
=(e,o,o,e) &\Rightarrow \bfk={\bf0} \text{ or } \bfe_1;\\
(\ell_{11},\ell_{12},\ell_{21},\ell_{22})
=(o,e,o,o) &\Rightarrow \bfk={\bf0} \text{ or } \bfe_1;\\
(\ell_{11},\ell_{12},\ell_{21},\ell_{22})
=(o,o,e,o) &\Rightarrow \bfk={\bf0} \text{ or } \bfe_2;\\
(\ell_{11},\ell_{12},\ell_{21},\ell_{22})
=(o,e,e,o) &\Rightarrow \bfk={\bf0},\bfe_1,\bfe_2 \text{ or }\bfe_1+\bfe_2.
\end{align*}
Here we denote by $o$ an odd integer and by $e$ an even integer.
{Recall that every automorphism $\varphi$ on $\Pi_1(\bfk)$ is determined by some $\bfn_i,\bfp,\bfx\in\bbz^2$ satisfying the conditions in Lemma~\ref{ch1}.} Now, given $A$ and $\bfk$, we will discuss the conditions on $A$ and on the automorphisms $\varphi$ on $\Pi_1(\bfk)$ for which $R(\varphi)=8$.

\begin{theorem} Let $A$ be the defining matrix of the pure lattice subgroup $\GammaA$ of $\Pi_1(\bfk)$. Let $\varphi:\Pi_1(\bfk)\to\Pi_1(\bfk)$ be any automorphism so that $\varphi'=\varphi|_{\GammaA}$ is of type $\mathrm{(II)}$ and $\det\varphi'=-1$. Then we have:
\begin{enumerate}
\item[$(1)$] When $\tr A$ is odd, $\bfk={\bf0}$ and $R(\varphi)=8$.
\item[$(2)$] When both $\ell_{11}$ and $\ell_{22}$ are even, $\bfk={\bf0}$ or $\bfe_1$ and $R(\varphi)=8$.
\item[$(3)$] When both $\ell_{11}, \ell_{21}$ and $\ell_{22}$ are odd and $\ell_{12}$ is even, $\bfk={\bf0}$ or $\bfe_1$ and $R(\varphi)=8$.
\item[$(4)$] When both $\ell_{11}, \ell_{12}$ and $\ell_{22}$ are odd and $\ell_{21}$ is even, $\bfk={\bf0}$ or $\bfe_2$ and $R(\varphi)=8$.
\item[$(5)$] When both $\ell_{11}, \ell_{22}$ are odd and both $\ell_{12}, \ell_{21}$ are even, then $\bfk={\bf0},\bfe_1,\bfe_2$ or $\bfe_1+\bfe_2$. Moreover, $R(\varphi)=8$ if and only if $\varphi$ is determined by the following conditions: For $\bfx=(x,y)^t$,
\begin{itemize}
\item when $\bfk={\bf0}$, $\bfx\in\im(I-[\varphi])$;
\item when $\bfk=\bfe_1$, either
\begin{align*}
[\varphi]=\left[\begin{matrix}-o&e\\\hspace{8pt}e&o\end{matrix}\right]
\text{ or } [\varphi]=\left[\begin{matrix}-o&o\\\hspace{8pt}e&o\end{matrix}\right] \text{ with $y$ even;}
\end{align*}
\item when $\bfk=\bfe_2$, either
\begin{align*}
[\varphi]=\left[\begin{matrix}-o&e\\\hspace{8pt}e&o\end{matrix}\right]
\text{ or } [\varphi]=\left[\begin{matrix}-o&e\\\hspace{8pt}o&o\end{matrix}\right] \text{ with $x$ even;}
\end{align*}
\item when $\bfk=\bfe_1+\bfe_2$, either
\begin{align*}\hspace{2cm}
[\varphi]=\left[\begin{matrix}-o&e\\\hspace{8pt}e&o\end{matrix}\right]
\text{ or } [\varphi]=\left[\begin{matrix}-e&o\\\hspace{8pt}o&e\end{matrix}\right] \text{ with
$\bfx=(e,e)^t$ or $(o,o)^t$.}
\end{align*}
\end{itemize}
\end{enumerate}
\end{theorem}

\begin{proof}
Remark that $\det(I-A)=2-\tr A$. Consider first the case where $\tr A$ is odd.
Then $\ell_{12}$ and $\ell_{21}$ are both odd, and $\bfk={\bf0}$.
In this case, we have the condition $2\bfp=(I-A^{-1})\bfx$.
From this identity, we can conclude that
$\bfx\in2(\bbz^2)$ and hence $(I-[\varphi])\bfx\in2(\bbz^2)$.

Consider next the case where $\tr A=\ell_{11}+\ell_{22}$ is even.
We divide into two cases.

{\sc Case 1:} Both $\ell_{11}$ and $\ell_{22}$ are even. \newline
Then both $\ell_{11}$ and $\ell_{22}$ are odd, and $\bfk={\bf0}$ or $\bfe_1$.
If $\bfk={\bf0}$, then since all the entries of $I-A^{-1}$ are odd, the condition $2\bfp=(I-A^{-1})\bfx$ induces that {both entries of $\bfx$ are even or odd}. If both entries of $\bfx$ are even, then $(I-[\varphi])\bfx\in2(\bbz^2)$. Consider the case where both entries of $\bfx$ are odd. We claim that all the entries of $I-[\varphi]$ are either even or odd. This will imply that $(I-[\varphi])\bfx\in2(\bbz^2)$.
If $u$ is even, then since $\det[\varphi]=-1$, both $v$ and $v'$ are odd and hence all the entries of $I-[\varphi]$ are odd. This shows that $(I-[\varphi])\bfx\in2(\bbz^2)$. If $u$ is odd, then $v$ or $v'$ is even; we can see easily that if $v$ is even, then $v'$ must be even and so all the entries of $I-[\varphi]$ are even and hence $(I-[\varphi])\bfx\in2(\bbz^2)$; if $v$ is odd and $v'$ must be odd, which is impossible.

Assume $\bfk=\bfe_1$. Consider the condition
\begin{align*}
2\bfp&=(I-A^{-1})\bfx-(A^{-1}+[\varphi])\bfe_1\\
&=\left[\begin{matrix}1-\ell_{22}&\ell_{12}\\\ell_{21}&1-\ell_{11}
\end{matrix}\right]\left[\begin{matrix}x\\y\end{matrix}\right]
-\left[\begin{matrix}\hspace{8pt}\ell_{22}\\-\ell_{21}
\end{matrix}\right]-\left[\begin{matrix}-u\\\hspace{8pt}v
\end{matrix}\right].
\end{align*}
If $v$ is odd, then $\ell_{21}x+(1-\ell_{11})y$ must be even and so both $x$ and $y$ are either even or odd; if both are even then $(I-[\varphi])\bfx\in2(\bbz^2)$, and if both are odd then by the claim above we have $(I-[\varphi])\bfx\in2(\bbz^2)$. If $v$ is even, then $\ell_{21}x+(1-\ell_{11})y$ must be odd and so $u$ must be odd, and furthermore $v'$ must be even, which yields that all the entries of $I-[\varphi]$ are even and so $(I-[\varphi])\bfx\in2(\bbz^2)$.

{\sc Case 2:} Both $\ell_{11}$ and $\ell_{22}$ are odd. \newline
Then $(\ell_{12},\ell_{21})=(e,o), (o,e)$ or $(e,e)$.
When $(\ell_{12},\ell_{21})=(e,o)$, then $\bfk={\bf0}$ or $\bfe_1$ and $v'$ is even.
From the fact that $\det\varphi'=-1$, we have $u$ is odd. When $\bfk={\bf0}$, the condition $2\bfp=(I-A^{-1})\bfx$ induces that $x$ must be even. Hence $(I-[\varphi])\bfx\in2(\bbz^2)$.
When $\bfk=\bfe_1$,we have the relation $2\bfp=(I-A^{-1})\bfx-(A^{-1}+[\varphi])\bfe_1$, which induces that $x$ is odd if and only if $v$ is even. It follows that $(I-[\varphi])\bfx\in2(\bbz^2)$.
When $(\ell_{12},\ell_{21})=(o,e)$, we repeat the above argument verbatim.

We finally consider the case where $(\ell_{12},\ell_{21})=(e,e)$.
Then $\bfk={\bf0},\bfe_1,\bfe_2$ or $\bfe_1+\bfe_2$.
When $\bfk=\bfe_1+\bfe_2$, the condition $2\bfp=(I-A^{-1})\bfx-(A^{-1}+[\varphi])(\bfe_1+\bfe_2)$ induces that both $-u+v'$ and $u+v$ are odd.
Thus $(u,v,v')=(o,e,e)$ or $(e,o,o)$.
If $(u,v,v')=(o,e,e)$, then all the entries of $I-[\varphi]$ are even and so
$(I-[\varphi])\bfx\in2(\bbz^2)$.
Now suppose $(u,v,v')=(e,o,o)$. Then all the entries of $I-[\varphi]$ are odd. Thus
$(I-[\varphi])\bfx\in2(\bbz^2)$ if and only if both $x$ and $y$ are either even or odd.

{When $\bfk={\bf0}$,} for any $\bfx$, since all the elements of $I-A^{-1}$ are even there is a unique $\bfp$ such that the condition $2\bfp=(I-A^{-1})\bfx$ holds. Hence for any automorphism $\varphi$ on $\Pi_1({\bf0})$, if it is  determined by $\bfx\in\im(I-[\varphi])$, then $R(\varphi)=8$ and vice versa.

{When $\bfk=\bfe_1$,} the condition $2\bfp=(I-A^{-1})\bfx-(A^{-1}+[\varphi])\bfe_1$ induces that $u$ is odd and $v$ is even. If $v'$ is even, then $(I-[\varphi])\bfx\in2(\bbz^2)$; if $v'$ is odd, then $(I-[\varphi])\bfx\in2(\bbz^2)\Leftrightarrow y\in2\bbz$. Hence for any automorphism $\varphi$ on $\Pi_1(\bfe_1)$, $(I-[\varphi])\bfx\in2(\bbz^2)$ if and only if $v'$ is even or else $v'$ is odd and $y\in2\bbz$.

{When $\bfk=\bfe_2$,} we repeat the above argument verbatim.
Hence for any automorphism $\varphi$ on $\Pi_1(\bfe_2)$, $(I-[\varphi])\bfx\in2(\bbz^2)$ if and only if $v$ is even or else $v$ is odd and $x\in2\bbz$.
\end{proof}
\bigskip

\section{The SC-groups $\Pi_3(\bfk,\bfk')$}

Recall that
$$
\Pi_3(\bfk,\bfk')=\left\langle{ a_1, a_2, t, \beta \ \Big|
\begin{array}{l}
[a_1,a_2]=1,\ ta_it^{-1}=A(a_i),\\
\beta a_i\beta^{-1}=M(a_i),
\beta^2=\bfa^{\bfk},\ \beta t\beta^{-1}=\bfa^{\bfk'}t^{-1}
\end{array}
}\right\rangle,
$$
where
$$
M=\left[\begin{matrix}-1&m\\\hspace{8pt}0&1
\end{matrix}\right]
$$
and $MAM^{-1}=A^{-1}$, and
\begin{align*}
(\bfk,\bfk'-\bfk)\in \frac{\ker(I-M)}{\im(I+M)}\oplus\frac{\ker(A-M)}{\im(A^{-1}+M)}.
\end{align*}

\begin{lemma}\label{P3}
The subgroup $\GammaA=\langle a_1,a_2,t\rangle$ of $\Pi_3(\bfk,\bfk')$ is a characteristic subgroup.
\end{lemma}

\begin{proof}
Let $\varphi:\Pi_3(\bfk,\bfk')\to\Pi_3(\bfk,\bfk')$ be an automorphism. Every element of $\Pi_3(\bfk,\bfk')$ is of the form
$\bfa^\bfx t^z \beta^w$ with $w=0$ or $1$.

Because $[a_1,a_2]=1$, there are the following $4$ possibilities:
\begin{enumerate}
\item $\varphi(a_i)=\bfa^{\bfn_{i}}t^{z_i}$,
\item $\varphi(a_i)=\bfa^{\bfn_{i}}t^{z}\beta$,
\item $\varphi(a_1)=\bfa^{\bfn_{1}}, \varphi(a_2)=\bfa^{\bfn_2}t^{z}\beta$,
\item $\varphi(a_1)=\bfa^{\bfn_{1}}t^{z}\beta, \varphi(a_2)=\bfa^{\bfn_2}$.
\end{enumerate}
We will show the last three possibilities cannot occur.

Consider the possibility (2) $\varphi(a_i)=\bfa^{\bfn_{i}}t^z\beta$. Assume that $\varphi(t)=\bfa^\bfx t^v\beta^w$ with $w=0$ or $1$. Then, noting that $(\bfa^\bfx t^v\beta)^\text{even}=\bfa^*$ and $(\bfa^\bfx t^v\beta)^\text{odd}=\bfa^*t^v\beta$, we have
\begin{align*}
ta_it^{-1}=A(a_i) &\Rightarrow
(\bfa^\bfx t^v\beta^w)(\bfa^{\bfn_{i}}t^{z}\beta)(\bfa^\bfx t^v\beta^w)^{-1}
=\varphi(a_1^{\ell_{1i}}a_2^{\ell_{2i}})\\
&\Rightarrow \ell_{1i}+\ell_{2i} \text{ is odd and so } \bfa^* t^{2v\pm z}\beta=\bfa^*t^{z}\beta.
\end{align*}
Thus $\varphi(t)=\bfa^\bfx$ or $\varphi(t)=\bfa^\bfx t^z\beta$. By Remark~\ref{FI}, there is a fully invariant, finite index, subgroup $\Lambda$ of $\Pi_3(\bfk,\bfk')$ such that $\Lambda\subset\GammaA$. By Remark~\ref{FI}, $\Lambda$ is generated by some elements of the form $\bfa^{\bfm_1}, \bfa^{\bfm_2}$ and $\bfa^*t^k$. Since $\varphi(\Lambda)\subset\Lambda$, and since $\langle\bfa^{\bfm_1}, \bfa^{\bfm_2}\rangle$ is a fully invariant subgroup of $\Lambda$, $\varphi$ induces an automorphism on $\Lambda/\langle\bfa^{\bfm_1}, \bfa^{\bfm_2}\rangle\cong\bbz$.  It follows from the above observation that $\varphi(t)$ cannot be $\bfa^\bfx$ or $\bfa^\bfx t^z\beta$.

Consider the possibility (3) $\varphi(a_1)=\bfa^{\bfn_{1}}, \varphi(a_2)=\bfa^{\bfn_2}t^{z}\beta$.
Assume that $\varphi(t)=\bfa^\bfx t^v\beta^w$. Then we have
\begin{align*}
ta_1t^{-1}=A(a_1) &\Rightarrow \text{$\ell_{21}$ is even and so both $\ell_{11}$ and $\ell_{22}$ are odd},\\
ta_2t^{-1}=A(a_2) &\Rightarrow {2v\pm z_2}=z_2 \text{ (since $\ell_{22}$ is odd) and hence}\\
&\hspace{0.65cm}\text{if $w=0$ then $v=0$ and if $w=1$ then $v=z_2$}.
\end{align*}
As before, any case cannot happen. In a similar way, the possibility (4) cannot occur.

Therefore we must have the only possibility (1) $\varphi(a_i)=a^{\bfn_{i}}t^{z_i}$. From  Remark~\ref{FI} again, since $\langle \bfa^{\bfm_1}, \bfa^{\bfm_2}\rangle$ is a finite index subgroup of $\langle a_1,a_2\rangle\cong\bbz^2$, we have $\det[\bfm_1\ \bfm_2]\ne0$. Furthermore,  $\langle \bfa^{\bfm_1}, \bfa^{\bfm_2}\rangle$ is fully invariant, and so $\varphi(\bfa^{\bfm_i})\in \langle \bfa^{\bfm_1}, \bfa^{\bfm_2}\rangle$. It follows that $z_1=z_2=0$, i.e., $\varphi(a_i)=\bfa^{\bfn_i}$. This shows that the subgroup $\langle a_1,a_2\rangle$ is a characteristic subgroup of $\Pi_3(\bfk,\bfk')$, so $\varphi$ induces an automorphism on the quotient group $\Pi_3(\bfk,\bfk')/\langle a_1,a_2\rangle\cong\bbz\rtimes\bbz_2$. This implies that
$\varphi(t)=\bfa^\bfp t^{\pm1}$ and $\varphi(\beta)=\bfa^\bfx t^z\beta$.
In particular, $\GammaA$ is a characteristic subgroup of $\Pi_3(\bfk,\bfk')$. Denote by $\varphi'$ the restriction of $\varphi$ on $\GammaA$, and denote $[\varphi']=[\bfn_1\ \bfn_2]$.
\end{proof}

\begin{theorem}
The SC-groups $\Pi_3(\bfk,\bfk')$ have the $R_\infty$ property.
\end{theorem}

\begin{proof}
For any automorphism $\varphi$ on $\Pi_3(\bfk,\bfk')$, by writing $\varphi'=\varphi|_{\GammaA}$, we have
$$
R(\varphi)\ge\frac{1}{2}\left(R(\varphi')+R(\tau_\beta\varphi')\right).
$$
Observe that
$$
\tau_\beta\varphi'(a_i)=\bfa^{M\bfn_i},\
\tau_\beta\varphi'(t)=\bfa^{M\bfp}(\bfa^{\bfk'}t^{-1})^{\pm1}.
$$
This implies that
$\varphi'$ is of type (II) if and only if $\tau_\beta\varphi'$ is of type (I).
Consequently, the theorem follows from Theorem~\ref{GA}.
\end{proof}
\bigskip

\section{The SC-groups $\Pi_4(\bfk)$}

Recall that
$$
\Pi_4(\bfk)=\left\langle{ a_1, a_2, \alpha, \beta \ |
\begin{array}{l}
[a_1,a_2]=1,\ \alpha a_i\alpha^{-1}=a_i^{-1}, \ \alpha^2=1,\\
\beta a_i\beta^{-1}=N(a_i),\ \alpha\beta \alpha^{-1}=\bfa^{\bfk}\beta
\end{array}}\right\rangle,
$$
where $A$ has a square root $N$
\begin{align*}
N=- \left[\begin{matrix}\tfrac{\ell_{11}-1}{\sqrt{\ell_{11}+\ell_{22}-2}}
&\tfrac{\ell_{12}}{\sqrt{\ell_{11}+\ell_{22}-2}}\\\tfrac{\ell_{21}}{\sqrt{\ell_{11}+\ell_{22}-2}}
&\tfrac{\ell_{22}-1}{\sqrt{\ell_{11}+\ell_{22}-2}}\end{matrix}\right]
\end{align*}
and
$$
\bfk\in \frac{\bbz^2}{\left(2(\bbz^2)+\im(I-N)\right)}.
$$

\begin{lemma}
The subgroup $\GammaA=\langle a_1,a_2,\beta^2\rangle$ of $\Pi_4(\bfk)$ is a characteristic subgroup.
\end{lemma}

\begin{proof}
Let $\varphi:\Pi_4(\bfk)\to\Pi_4(\bfk)$ be an automorphism. Every element of $\Pi_4(\bfk)$ is of the form $\bfa^\bfx \beta^z\alpha^w$ with $w=0$ or $1$.

Because $\alpha^2=1$, it follows that {$\varphi(\alpha)=\bfa^{\bfx}\alpha$}.
If $\varphi(a_i)=\bfa^{\bfn_{i}}\beta^m\alpha^w$, then $\alpha a_i\alpha^{-1}=a_i^{-1} \Rightarrow m=w=0$. Thus {$\varphi(a_i)=\bfa^{\bfn_{i}}$}. In particular, we have shown that the subgroup $\langle a_1,a_2,\alpha\rangle\subset\Pi_4(\bfk)$ is characteristic and hence $\varphi$ induces an {automorphism} on the quotient group $\Pi_4(\bfk)/\langle a_1,a_2,\alpha\rangle\cong\bbz$. This implies that $\varphi(\beta)=\bfa^{\bfp}\beta^{\pm1}\alpha^w$ and thus $\varphi(\beta^2)=(\bfa^{\bfp}\beta^{\pm1}\alpha^w)^2=\bfa^*\beta^{\pm2}$.

Consequently, we have shown that $\GammaA$ is a characteristic subgroup of $\Pi_4(\bfk)$.
\end{proof}


\begin{theorem}
The SC-groups $\Pi_4(\bfk)$ have the $R_\infty$ property.
\end{theorem}

\begin{proof}
For any automorphism $\varphi$ on $\Pi_4(\bfk)$, by writing $\varphi'=\varphi|_{\GammaA}$, we have
$$
R(\varphi)\ge\frac{1}{4}\left(R(\varphi')+R(\tau_\alpha\varphi')+R(\tau_\beta\varphi')
+R(\tau_{\alpha\beta}\varphi')\right).
$$
Observe further that
\begin{align*}
&\tau_\alpha\varphi'(a_i)=\bfa^{-\bfn_{i}},\
\tau_\alpha\varphi'(\beta^2)=\bfa^*(\beta^{2})^{\pm1},\\
&\tau_\beta\varphi'(a_i)=\bfa^{N\bfn_i},\
\tau_\beta\varphi'(\beta^2)=\bfa^*(\beta^{2})^{\pm1},\\
&\tau_{\alpha\beta}\varphi'(a_i)=\bfa^{-N\bfn_{i}},\
\tau_{\alpha\beta}\varphi'(\beta^2)=\bfa^*(\beta^{2})^{\pm1}.
\end{align*}
Hence $\varphi', \tau_\alpha\varphi', \tau_\beta\varphi'$ and $\tau_{\alpha\beta}\varphi'$ have the same types and $\det\varphi'=\det\tau_\alpha\varphi'
=-\det\tau_\beta\varphi'=-\det\tau_{\alpha\beta}\varphi'$ because $\det N=-1$.
This implies from Theorem~\ref{GA} that when $\varphi'$ is of type (II),
$$R(\varphi')=4 \Leftrightarrow \det\varphi'=-1 \Leftrightarrow \det\tau_\beta\varphi'=1
\Leftrightarrow R(\tau_\beta\varphi')=\infty.$$
Consequently, the theorem is proved.
\end{proof}
\bigskip

\section{The SC-groups $\Pi_5(\bfm,\bfk,\bfk',\bfn)$}

Recall that
$$
\Pi_5(\bfm,\bfk,\bfk',\bfn) =\left\langle{ a_1, a_2, t, \alpha, \beta \ \Big|
\begin{array}{l}
[a_1,a_2]=1, \ ta_it^{-1}=A(a_i),\\
\alpha a_i\alpha^{-1}=a_i^{-1},\ \beta a_i\beta^{-1}=M(a_i),\\
\alpha^2=1,\ \beta^2=\bfa^{\bfk},\ [\alpha,\beta]=\bfa^{\bfn}\\
\alpha t\alpha^{-1}=\bfa^{\bfm}t,\
\beta t\beta^{-1}=\bfa^{\bfk'}t^{-1}
\end{array}}\right\rangle,
$$
where $M$ is traceless with determinant $-1$ and $MAM^{-1}=A^{-1}$, and
\begin{align*}
&(\bfk,\bfn+\bfk,\bfk'-\bfk,\bfm-\bfn+M(\bfk'-\bfk))\\
&\quad\in \frac{\ker(I-M)}{\im(I+M)}\oplus\frac{\ker(I+M)}{\im(I-M)}\\
&\qquad\oplus\frac{\ker(I-A^{-1}M)}{\eta\x\im(I+MA)}\oplus\frac{\ker(A^{-1}+M)}{\im(A-M)}.
\end{align*}
Here $\eta=1$ or $2$, and $\eta=2$ if and only if $A$ and $M$ can be conjugated simultaneously to
$$
\left[\begin{matrix}\ell'_{11}&\ell'_{12}\\\ell'_{21}&\ell'_{22}\end{matrix}\right]\ \text{ and }\
\left[\begin{matrix}-1&1\\\hspace{8pt}0&1\end{matrix}\right]
$$
so that both $\tfrac{\ell'_{21}}{\gcd(\ell'_{22}-1,\ell'_{21})}$ and $\tfrac{\ell'_{21}}{\gcd(\ell'_{22}+1,\ell'_{21})}$ are even.

\begin{lemma}
The subgroups $\langle a_1,a_2,t^2\rangle\subset\langle a_1,a_2,t,\alpha\rangle$ of $\Pi_5(\bfm,\bfk,\bfk',\bfn)$ are characteristic subgroups.
\end{lemma}

\begin{proof}
Write $\Pi_5=\Pi_5(\bfm,\bfk,\bfk',\bfn)$. Let $\varphi:\Pi_5\to\Pi_5$ be an automorphism. Every element of $\Pi_5$ can be written uniquely as the form of $\bfa^\bfx t^z\alpha^v\beta^w$ with $v,w\in\{0,1\}$.

Because $\alpha^2=1$, it follows that $\varphi(\alpha)=\bfa^{\bfx}\alpha$ or $\bfa^\bfx t^z\alpha^v\beta$.
Because $[a_1,a_2]=1$ and using the fact that the elements of the form $\bfa^\bfx\alpha$ are torsion elements of order $2$, we can derive the following possibilities:
\begin{enumerate}
\item $\varphi(a_i)=\bfa^{\bfn_{i}}t^{z_i}\alpha^{v_i}$,
\item $\varphi(a_i)=\bfa^{\bfn_{i}}t^{z}\alpha^{v_i}\beta$,
\item $\varphi(a_1)=\bfa^{\bfn_{1}}, \varphi(a_2)=\bfa^{\bfn_2}t^{z}\alpha^{v_2}\beta$,
\item $\varphi(a_1)=\bfa^{\bfn_{1}}t^{z}\alpha^{v_1}\beta, \varphi(a_2)=\bfa^{\bfn_2}$.
\end{enumerate}
Observe that $(\bfa^\bfx t^z\alpha^v\beta)^\text{even}=\bfa^*$ and $(\bfa^\bfx t^z\alpha^v\beta)^\text{odd}=\bfa^*t^z\alpha^v\beta$.

Consider the possibility (2) $\varphi(a_i)=\bfa^{\bfn_{i}}t^{z}\alpha^{v_i}\beta$. Assume $\varphi(t)=\bfa^\bfx t^u\alpha^v\beta^w$. Then
$ta_it^{-1}=A(a_i) \Rightarrow \ell_{1i}+\ell_{2i} $ is odd and so
$\bfa^*t^{2u+(-1)^wz}\alpha^v\beta=\bfa^*t^z\alpha^{v_{i'}}\beta$, respectively.
Hence $\varphi(t)$ is $\bfa^\bfx\alpha^{v}$ or $\bfa^\bfx t^z\alpha^{v}\beta$. By Remark~\ref{FI}, $\Pi_5$ has a fully invariant subgroup $\Lambda\subset\GammaA$ and hence $\varphi$ induces an automorphism on the group $\Lambda/\Lambda\cap\langle a_1,a_2\rangle\cong\bbz$. This rules out the case $\varphi(t)=\bfa^\bfx\alpha^{v}$. Furthermore, the above observation rules out the other case $\varphi(t)=\bfa^\bfx t^z\alpha^{v}\beta$.

Consider the possibility (3)
$\varphi(a_1)=\bfa^{\bfn_{1}}, \varphi(a_2)=\bfa^{\bfn_2}t^{z}\alpha^{v_2}\beta$.
Assume $\varphi(t)=\bfa^\bfx t^u\alpha^v\beta^w$. Then
\begin{align*}
ta_1t^{-1}=A(a_1)&\Rightarrow \ell_{21} \text{ is even and so both $\ell_{11}$ and $\ell_{22}$ are odd},\\
ta_2t^{-1}=A(a_2)&\Rightarrow 2u+(-1)^wz=z
\text{ (since $\ell_{22}$ is odd).}
\end{align*}
Hence $\varphi(t)$ is of the form $\bfa^\bfx\alpha^v$ or
$\bfa^\bfx t^{z}\alpha^v\beta$.
As above, any case cannot occur. Similarly, the possibility (4) cannot occur.

Therefore, we have the only possibility (1) $\varphi(a_i)=\bfa^{\bfn_{i}}t^{z_i}\alpha^{v_i}$. From  Remark~\ref{FI}, since $\langle \bfa^{\bfm_1}, \bfa^{\bfm_2}\rangle$ is a finite index subgroup of $\langle a_1,a_2\rangle\cong\bbz^2$, we have $\det[\bfm_1\ \bfm_2]\ne0$. Furthermore,  $\langle \bfa^{\bfm_1}, \bfa^{\bfm_2}\rangle$ is fully invariant, and so $\varphi(\bfa^{\bfm_i})\in \langle \bfa^{\bfm_1}, \bfa^{\bfm_2}\rangle$. Since
\begin{align*}
\varphi(\bfa^{\bfm_i})&=\varphi(a_1)^{m_{1i}}\varphi(a_2)^{m_{2i}}
=(\bfa^{\bfn_1}t^{z_1}\alpha^{v_1})^{m_{1i}}(\bfa^{\bfn_{2}}t^{z_2}\alpha^{v_2})^{m_{2i}}\\
&=\bfa^* t^{m_{1i}z_1+m_{2i}z_2}\alpha^{m_{1i}v_1+m_{2i}v_2},
\end{align*}
we have $m_{1i}z_1+m_{2i}z_2=0$ and $m_{1i}v_1+m_{2i}v_2$ is even. Because the matrix $[\bfm_1\ \bfm_2]$ is nonsigular, $z_1=z_2=0$ and hence $\varphi(a_i)=\bfa^{\bfn_{i}}\alpha^{v_i}$. If $v_i=1$ then $\bfa^{\bfn_{i}}\alpha^{v_i}$ is of order $2$. This shows that $\varphi(a_i)=\bfa^{\bfn_{i}}$.

  This shows that the subgroup $\langle a_1,a_2\rangle$ is a characteristic subgroup of $\Pi_5$, so $\varphi$ induces an automorphism $\bar\varphi$ on the quotient group $\Pi_5/\langle a_1,a_2\rangle$, which is isomorphic to
$$
\langle\bar{t},\bar\alpha,\bar\beta\mid \bar\alpha^2=\bar\beta^2=1, [\bar\alpha,\bar\beta]=1, \bar\alpha\bar{t}\bar\alpha^{-1}=\bar{t}, \bar\beta\bar{t}\bar\beta^{-1}=\bar{t}^{-1}\rangle.
$$
Recall that $\varphi(\alpha)=\bfa^\bfx\alpha$ or $\bfa^\bfx t^z\alpha^v\beta$. Consider the case {$\varphi(\alpha)=\bfa^\bfx\alpha$}. Then $\bar\varphi(\bar\alpha)=\bar\alpha$. Since $\bar{t}$ is a torsion-free element, $\bar\varphi(\bar{t})$ is a torsion-free element, say $\bar{t}^m\bar\alpha^v$. Because $\bar\varphi$ is an automorphism fixing $\bar\alpha$, we must have $\bar\varphi(\bar{t})=\bar{t}^{\pm1}\bar\alpha^v$ and $\bar\varphi(\bar\beta)=\bar{t}^z\bar\alpha^{v'}\bar\beta$. This implies that {$\varphi(t)=\bfa^\bfp t^{\pm1}\alpha^v$} and $\varphi(\beta)=\bfa^\bfq t^z\alpha^{v'}\beta$. Consider next the case $\varphi(\alpha)=\bfa^\bfx t^z\alpha^v\beta$. So, $\bar\varphi(\bar\alpha)=\bar{t}^z\bar\alpha^v\bar\beta$. As $\bar\varphi(\bar{t})$ is of the form $\bar{t}^m\bar\alpha^{v'}$ with $m\ne0$, we have
\begin{align*}
\bar{t}\bar\alpha=\bar\alpha\bar{t}
&\Rightarrow (\bar{t}^m\bar\alpha^{v'})(\bar{t}^z\bar\alpha^v\bar\beta)
=(\bar{t}^z\bar\alpha^v\bar\beta)(\bar{t}^m\bar\alpha^{v'})\\
&\Rightarrow \bar{t}^{m+z}\bar\alpha^{v'+v}\bar\beta=\bar{t}^{z-m}\bar\alpha^{v+v'}\bar\beta
\Rightarrow m=0.
\end{align*}
Thus this case cannot occur.

In all, we have shown that $\varphi$ is of the form
\begin{align*}
&\varphi(a_i)=\bfa^{\bfn_i},\ \varphi(t)=\bfa^{\bfp}t^{\pm1}\alpha^v,\\
&\varphi(\alpha)=\bfa^\bfx\alpha, \ \varphi(\beta)=\bfa^\bfq t^z\alpha^{v'}\beta.
\end{align*}
Consequently, the subgroups $\langle a_1,a_2,t^2\rangle\subset\langle a_1,a_2,t,\alpha\rangle$ are characteristic subgroups of $\Pi_5$.
\end{proof}

\begin{theorem}
The SC-groups $\Pi_5(\bfm,\bfk,\bfk',\bfn)$ have the $R_\infty$ property.
\end{theorem}

\begin{proof}
Note that the subgroup $\langle a_1,a_2,t^2\rangle$ is a lattice of $\Sol$ determined by the matrix $A^2$. We denote this group by $\Gamma_{\!A^2}$. Then $\Pi_5(\bfm,\bfk,\bfk',\bfn)/\Gamma_{\!A^2}\cong(\bbz_2)^3$. Let $\varphi$ be an automorphism on $\Pi_5(\bfm,\bfk,\bfk',\bfn)$. With $\varphi'=\varphi|_{\Gamma_{\!A^2}}$, by Theorem~\ref{HLP} we have
\begin{align*}
R(\varphi)\ge\frac{1}{8}&\left(R(\varphi')+R(\tau_\alpha\varphi')
+R(\tau_\beta\varphi')+R(\tau_{\alpha\beta}\varphi')\right.\\
+&\left.R(\tau_{t}\varphi')+R(\tau_{t\alpha}\varphi')+R(\tau_{t\beta}\varphi')
+R(\tau_{t\alpha\beta}\varphi')\right).
\end{align*}
Observe that
\begin{align*}
&\tau_\beta\varphi'(a_i)=\beta\bfa^{\bfn_i}\beta^{-1}=\bfa^{M\bfn_i},\\
&\tau_\beta\varphi'(t^2)=\beta(\bfa^\bfp t^{\pm1}\alpha^v)^2\beta^{-1}=\bfa^*t^{\mp2}.
\end{align*}
This shows that
$\varphi'$ is of type (II) if and only if $\tau_\beta\varphi'$ is of type (I).
Consequently, the theorem follows from Theorem~\ref{GA}.
\end{proof}
\bigskip

\section{The SC-groups $\Pi_6(\bfk,\bfk')$}

Recall that
$$
\Pi_6(\bfk,\bfk')=\left\langle{a_1, a_2, \alpha, \beta \ \Big|
\begin{array}{l}
[a_1,a_2]=1,\\
\alpha a_i\alpha^{-1}=N(a_i),\
\beta a_i\beta^{-1}=M(a_i),\\
\beta^2=\bfa^{\bfk},\
\beta \alpha\beta^{-1}=\bfa^{\bfk'}\alpha^{-1}
\end{array}}\right\rangle,
$$
where $A$ has a square root $N$
\begin{align*}
N=-
\left[\begin{matrix}\tfrac{\ell_{11}+1}{\sqrt{\ell_{11}+\ell_{22}+2}}
&\tfrac{\ell_{12}}{\sqrt{\ell_{11}+\ell_{22}+2}}\\\tfrac{\ell_{21}}{\sqrt{\ell_{11}+\ell_{22}+2}}
&\tfrac{\ell_{22}+1}{\sqrt{\ell_{11}+\ell_{22}+2}}\end{matrix}\right],
\end{align*}
and $M$ is traceless with determinant $-1$ and $MAM^{-1}=A^{-1}$, and
\begin{align*}
(\bfk,\bfk'-\bfk)\in \frac{\ker(I-M)}{\im(I+M)}\oplus\frac{\ker(N-M)}{\im(N^{-1}+M)}.
\end{align*}

\begin{lemma}
The subgroup $\GammaA=\langle a_1,a_2,\alpha^2\rangle$ of $\Pi_6(\bfk,\bfk')$ is a characteristic subgroup.
\end{lemma}

\begin{proof}
Every element of $\Pi_6(\bfk,\bfk')$ is of the form $\bfa^\bfx\alpha^z\beta^w$ with $w=0$ or $1$. Let $\varphi:\Pi_6(\bfk,\bfk')\to\Pi_6(\bfk,\bfk')$ be an automorphism. Because $[a_1,a_2]=1$, there are the following $4$ possibilities:
\begin{enumerate}
\item $\varphi(a_i)=\bfa^{\bfn_{i}}\alpha^{z_i}$,
\item $\varphi(a_i)=\bfa^{\bfn_{i}}\alpha^{z}\beta$,
\item $\varphi(a_1)=\bfa^{\bfn_{1}}, \varphi(a_2)=\bfa^{\bfn_2}\alpha^{z}\beta$,
\item $\varphi(a_1)=\bfa^{\bfn_{1}}\alpha^{z}\beta, \varphi(a_2)=\bfa^{\bfn_2}$.
\end{enumerate}
Observe that $(\bfa^\bfx\alpha^z\beta)^\text{even}=\bfa^*$ and $(\bfa^\bfx\alpha^z\beta)^\text{odd}=\bfa^*\alpha^z\beta$.

Consider the possibility (2) $\varphi(a_i)=\bfa^{\bfn_{i}}\alpha^{z}\beta$. Let $\varphi(\alpha)=\bfa^\bfx \alpha^u\beta^w$. Then
$\alpha a_i\alpha^{-1}=N(a_i)$ induces that
$\ell_{1i}'+\ell_{2i}'$ is odd
where $N=(\ell_{ij}')$, and  so $\bfa^*\alpha^{2u+(-1)^wz}\beta=\bfa^*\alpha^z\beta$.
Hence $\varphi(\alpha)$ is $\bfa^\bfx$ or $\bfa^\bfx\alpha^{z}\beta$.
By Remark~\ref{FI}, $\Pi_6(\bfk,\bfk')$ has a fully invariant subgroup $\Lambda\subset\GammaA$ and hence $\varphi$ induces an automorphism on the group $\Lambda/\Lambda\cap\langle a_1,a_2\rangle\cong\bbz$. This rules out the case $\varphi(\alpha)=\bfa^\bfx$. Furthermore, the above observation rules out the other case $\varphi(\alpha)=\bfa^\bfx\alpha^z\beta$.

Consider the possibility (3)
$\varphi(a_1)=\bfa^{\bfn_{1}}, \varphi(a_2)=\bfa^{\bfn_2}\alpha^{z}\beta$.
Let $\varphi(\alpha)=\bfa^\bfx \alpha^u\beta^w$. Then
\begin{align*}
\alpha a_1\alpha^{-1}=N(a_1)&\Rightarrow \ell_{21}' \text{ is even and so both $\ell_{11}'$ and $\ell_{22}'$ are odd},\\
\alpha a_2\alpha^{-1}=N(a_2)&\Rightarrow 2u+(-1)^wz_2=z_2 \text{ (since $\ell_{22}'$ is odd).}
\end{align*}
Hence $\varphi(\alpha)$ is of the form $\bfa^\bfx$ or $\bfa^\bfx\alpha^{z_2}\beta$. As above, any case cannot occur. Similarly, the possibility (4) cannot occur.

Therefore, we have the only possibility (1) $\varphi(a_i)=\bfa^{\bfn_{i}}\alpha^{z_i}$. From  Remark~\ref{FI}, since $\langle \bfa^{\bfm_1}, \bfa^{\bfm_2}\rangle$ is a finite index subgroup of $\langle a_1,a_2\rangle\cong\bbz^2$, we have $\det[\bfm_1\ \bfm_2]\ne0$. Furthermore,  $\langle \bfa^{\bfm_1}, \bfa^{\bfm_2}\rangle$ is fully invariant, and so $\varphi(\bfa^{\bfm_i})\in \langle \bfa^{\bfm_1}, \bfa^{\bfm_2}\rangle$. Since
\begin{align*}
\varphi(\bfa^{\bfm_i})&=\varphi(a_1)^{m_{1i}}\varphi(a_2)^{m_{2i}}
=(\bfa^{\bfn_1}\alpha^{z_1})^{m_{1i}}(\bfa^{\bfn_{2}}\alpha^{z_2})^{m_{2i}}
=\bfa^*\alpha^{m_{1i}z_1+m_{2i}z_2},
\end{align*}
we have $m_{1i}z_1+m_{2i}z_2=0$. As the matrix $[\bfm_1\ \bfm_2]$ is nonsigular, $z_1=z_2=0$ and hence $\varphi(a_i)=\bfa^{\bfn_{i}}$.
This shows that the subgroup $\langle a_1,a_2\rangle$ is a characteristic subgroup of $\Pi_6(\bfk,\bfk')$, so $\varphi$ induces an automorphism $\bar\varphi$ on the quotient group $\Pi_6(\bfk,\bfk')/\langle a_1,a_2\rangle\cong\bbz\rtimes\bbz_2$, with generators $\bar\alpha,\bar\beta$. Hence we must have $\bar\varphi(\bar\alpha)=\bar\alpha^{\pm1}$ and $\bar\varphi(\bar\beta)=\bar\alpha^{v'}\bar\beta$, or $\varphi(\alpha)=\bfa^\bfp\alpha^{\pm1}$ and $\varphi(\beta)=\bfa^\bfq\alpha^{v'}\beta$.

Consequently, the subgroups $\GammaA=\langle a_1,a_2,\alpha^2\rangle\subset\langle a_1,a_2,\alpha\rangle$ are characteristic subgroups of $\Pi_6(\bfk,\bfk')$.
\end{proof}

\begin{theorem}
The SC-groups $\Pi_6(\bfk,\bfk')$ have the $R_\infty$ property.
\end{theorem}

\begin{proof}
Let $\varphi$ be an automorphism on $\Pi_6(\bfk,\bfk')$ and let $\varphi'$ be the restriction of $\varphi$ on ${\GammaA}$. Since $\Pi_6(\bfk,\bfk')/\GammaA\cong(\bbz_2)^2$, by Theorem~\ref{HLP}, we have
$$
R(\varphi)\ge\frac{1}{4}\left(
R(\varphi')+R(\tau_\alpha\varphi')+R(\tau_\beta\varphi')+R(\tau_{\alpha\beta}\varphi')
\right).
$$
Since
$$
\tau_\beta\varphi(a_i)=\beta\bfa^{\bfn_i}\beta^{-1}=\bfa^{M\bfn_i},\
\tau_\beta\varphi(\alpha^2)=\bfa^*\alpha^{\mp2},
$$
it follows that $\varphi'$ is of type (II) if and only if $\tau_\beta\varphi'$ is of type (I). By Theorem~\ref{GA}, we have the result.
\end{proof}
\bigskip

\section{The SC-groups $\Pi_{7}(\bfk)$}

Recall that
$$
\Pi_{7}(\bfk)=\left\langle{a_1, a_2, t,\alpha \ \Big|
\begin{array}{l}
[a_1,a_2]=1,\ ta_it^{-1}=A(a_i),\ \alpha a_i\alpha^{-1}=M(a_i),\\
\alpha^4=1,\ \alpha t\alpha^{-1}=\bfa^{\bfk}t^{-1}
\end{array}}\right\rangle,
$$
where $M$ is traceless with determinant $1$ and $MAM^{-1}=A^{-1}$ and
\begin{align*}
\bfk\in \frac{\bbz^2}{\left(\im(M+A^{-1})+\im(I-A^{-1})\right)}.
\end{align*}
Consider the subgroup of $\Pi_7(\bfk)$ generated by $a_1,a_2,t$ and $\alpha^2$. Since $\alpha t\alpha^{-1}=\bfa^\bfk t^{-1}$, we have
$$
\alpha^2t\alpha^{-2}=\alpha(\bfa^\bfk t^{-1})\alpha^{-1}=\bfa^{(M-A)\bfk}t.
$$
Note also that the condition on $M$ above implies that $M^2=-I$. Hence the subgroup $\langle a_1,a_2,t,\alpha^2\rangle$ is $\Pi_1((M-A)\bfk)$.

\begin{lemma}
The subgroup $\Pi_1((M-A)\bfk)$ of $\Pi_7(\bfk)$ is a characteristic subgroup.
\end{lemma}

\begin{proof}
Every element of $\Pi_7(\bfk)$ is of the form $\bfa^\bfx t^z\alpha^w$ with $w=0,1,2$ or $3$. Let $\varphi:\Pi_7(\bfk)\to\Pi_7(\bfk)$ be an automorphism. Because $[a_1,a_2]=1$
and using the fact that the elements of the form $\bfa^\bfx\alpha^2$ are torsion elements of order $2$, we can derive the following possibilities:
\begin{enumerate}
\item $\varphi(a_i)=\bfa^{\bfn_{i}}t^{z_i}\alpha^{w_i}$ with $w_i\in\{0,2\}$,
\item $\varphi(a_i)=\bfa^{\bfn_{i}}t^z\alpha^{w_i}$ with $w_i\in\{1,3\}$,
\item $\varphi(a_1)=\bfa^{\bfn_{1}}$, $\varphi(a_2)=\bfa^{\bfn_2}t^{z}\alpha^{w_2}$ with $w_2\in\{1,3\}$,
\item $\varphi(a_1)=\bfa^{\bfn_{1}}t^{z}\alpha^{w_1}$,
$\varphi(a_2)=\bfa^{\bfn_2}$ with $w_1\in\{1,3\}$.
\end{enumerate}
Observe that $(\bfa^\bfx t^z\alpha^w)^e=\bfa^*t^{\frac{e}{2}(1+(-1)^w)z}\alpha^{ew}$ when $e$ is even.

Consider the possibility (2) $\varphi(a_i)=\bfa^{\bfn_{i}}t^{z}\alpha^{w_i}$ with $w_i\in\{1,3\}$. Let $\varphi(t)=\bfa^\bfx t^u\alpha^w$. Then
$ta_it^{-1}=A(a_i)$ induces that
$\ell_{1i}+\ell_{2i}$ is odd
and $\bfa^*t^{2u+(-1)^wz}\alpha^{w_i}=\bfa^*t^z\alpha^{w_1\ell_{1i}+w_2\ell_{2i}}$. This implies that if $w$ is even then $u=0$ and if $w$ is odd then $u=z$. Hence $\varphi(t)$ is $\bfa^\bfx, \bfa^\bfx\alpha^2, \bfa^\bfx t^z\alpha$ or $\bfa^\bfx t^z\alpha^3$.
By Remark~\ref{FI}, $\Pi_7(\bfk)$ has a fully invariant subgroup $\Lambda\subset\GammaA$ and hence $\varphi$ induces an automorphism on the group $\Lambda/\Lambda\cap\langle a_1,a_2\rangle\cong\bbz$. Since $\varphi(t^4)=\bfa^*$, this rules out all the cases of $\varphi(t)$.

Consider the possibility (3)
$\varphi(a_1)=\bfa^{\bfn_{1}}$, $\varphi(a_2)=\bfa^{\bfn_2}t^{z}\alpha^{w_2}$ with $w_2\in\{1,3\}$.
Let $\varphi(t)=\bfa^\bfx t^u\alpha^w$. Then
\begin{align*}
ta_1t^{-1}=A(a_1)&\Rightarrow \ell_{21} \text{ is even and so both $\ell_{11}$ and $\ell_{22}$ are odd},\\
ta_2t^{-1}=A(a_2)&\Rightarrow 2u+(-1)^wz=z
\text{ (since $\ell_{22}$ is odd).}
\end{align*}
Hence $\varphi(t)$ is $\bfa^\bfx, \bfa^\bfx\alpha^2, \bfa^\bfx t^z\alpha$ or $\bfa^\bfx t^z\alpha^3$. As above, any case cannot occur. Similarly, the possibility (4) cannot occur.

In all, we have the only possibility (1) $\varphi(a_i)=\bfa^{\bfn_{i}}t^{z_i}\alpha^{w_i}$ with $w_i\in\{0,2\}$.
From  Remark~\ref{FI}, since $\langle \bfa^{\bfm_1}, \bfa^{\bfm_2}\rangle$ is a finite index subgroup of $\langle a_1,a_2\rangle\cong\bbz^2$, we have $\det[\bfm_1\ \bfm_2]\ne0$. Furthermore,  $\langle \bfa^{\bfm_1}, \bfa^{\bfm_2}\rangle$ is fully invariant, and so $\varphi(\bfa^{\bfm_i})\in \langle \bfa^{\bfm_1}, \bfa^{\bfm_2}\rangle$. Since
\begin{align*}
\varphi(\bfa^{\bfm_i})&=\varphi(a_1)^{m_{1i}}\varphi(a_2)^{m_{2i}}
=(\bfa^{\bfn_1}t^{z_1}\alpha^{w_1})^{m_{1i}}(\bfa^{\bfn_{2}}t^{z_2}\alpha^{w_2})^{m_{2i}}\\
&=\bfa^*t^{m_{1i}z_1+m_{2i}z_2}\alpha^{m_{1i}w_1+m_{2i}w_2},
\end{align*}
we have $m_{1i}z_1+m_{2i}z_2=0$ and $m_{1i}w_1+m_{2i}w_2\equiv0\mod(4)$. As the matrix $[\bfm_1\ \bfm_2]$ is nonsigular, $z_1=z_2=0$ and hence $\varphi(a_i)=\bfa^{\bfn_{i}}\alpha^{w_i}$ with $w_i$ even. Since $\bfa^{\bfn_i}\alpha^2$ is an element of order $2$, we must have $w_i=0$  and $\varphi(a_i)=\bfa^{\bfn_{i}}$.

On the other hand, since $\alpha$ is a torsion element, so is $\varphi(\alpha)$ and this shows that $\varphi(\alpha)$ is of the form $\bfa^\bfx t^u\alpha^w$ with $w$ odd.

In all, the subgroup $\langle a_1,a_2,\alpha^2\rangle$ is a characteristic subgroup of $\Pi_7(\bfk)$, so $\varphi$ induces an automorphism $\bar\varphi$ on the quotient group $\Pi_7(\bfk)/\langle a_1,a_2,\alpha^2\rangle\cong\bbz\rtimes\bbz_2$, with generators $\bar{t}, \bar\alpha$. It is clear that $\bar\varphi(\bar{t})=\bar{t}^{\pm1}$. Summing up, we have
\begin{align*}
&\varphi(a_i)=\bfa^{\bfn_i},\\
&\varphi(t)=\bfa^\bfp t^{\pm1}\alpha^v \text{ with $v$ even},\\
&\varphi(\alpha)=\bfa^\bfx t^u\alpha^w \text{ with $w$ odd.}
\end{align*}
Therefore it follows easily that the subgroup $\Pi_1((M-A)\bfk)$ is a characteristic subgroup of $\Pi_7(\bfk)$.
\end{proof}

We remark from Lemma~\ref{ch1} that {$\GammaA=\langle a_1, a_2, t\rangle$} is a characteristic subgroup of $\Pi_1((M-A)\bfk)$ and $v$ must be $0$. Furthermore, $\GammaA$ is a characteristic subgroup of $\Pi_7(\bfk)$.

\begin{theorem}
The SC-groups $\Pi_7(\bfk)$ have the $R_\infty$ property.
\end{theorem}

\begin{proof}
Let $\varphi$ be an automorphism on $\Pi_7(\bfk)$ and let $\varphi'$ be the restriction of $\varphi$ on $\Pi_1((M-A)\bfk)$, which is of index $2$ in $\Pi_7(\bfk)$. By Theorem~\ref{HLP}, we have
$$
R(\varphi)\ge\frac{1}{2}\left(R(\varphi')+R(\tau_{\alpha}\varphi')\right).
$$
Since
$$
\tau_\alpha\varphi(a_i)=\alpha(\bfa^{\bfn_i})\alpha^{-1}=\bfa^{A\bfn_i},\
\tau_\alpha\varphi(t)=\bfa^*t^{\mp1},
$$
it shows that $\varphi'$ is of type (II) if and only if $\tau_\alpha\varphi'$ is of type (I). By Theorem~\ref{GA}, we have the result.
\end{proof}
\bigskip

\section{The SC-groups $\Pi_{8}(\bfk,\bfm)$}

Recall that
$$
\Pi_{8}(\bfk,\bfm)=\left\langle{ a_1, a_2, \alpha, \beta \ \Big|
\begin{array}{l}
[a_1,a_2]=1,\ \beta a_i\beta^{-1}=N(a_i),\\
\alpha a_i\alpha^{-1}=M(a_i),\ \alpha^4=1,\\
\alpha \beta^2 \alpha^{-1}=\bfa^{\bfk}\beta^{-2},\
\alpha\beta^{-1}=\bfa^{\bfm}\beta\alpha^{-1}
\end{array}}\right\rangle,
$$
where $A$ has a square root $N$
\begin{align*}
N=-
\left[\begin{matrix}\tfrac{\ell_{11}-1}{\sqrt{\ell_{11}+\ell_{22}-2}}
&\tfrac{\ell_{12}}{\sqrt{\ell_{11}+\ell_{22}-2}}\\\tfrac{\ell_{21}}{\sqrt{\ell_{11}+\ell_{22}-2}}
&\tfrac{\ell_{22}-1}{\sqrt{\ell_{11}+\ell_{22}-2}}\end{matrix}\right],
\end{align*}
and $M$ is traceless with determinant $1$ and $MAM^{-1}=A^{-1}$, and
\begin{align*}
&(\bfk,\bfm)\in \frac{\bbz^2}{\im\left((I-A^{-1})+(M+A^{-1})(I+N)\right)}\oplus\frac{\ker(M+N^{-1})}{\im(M+N)}.
\end{align*}
Notice that the subgroup $\langle a_1,a_2,\beta^2,\alpha\rangle$ of $\Pi_8(\bfk,\bfm)$ is isomorphic to $\Pi_7(\bfk)$, and $\langle a_1,a_2,\beta^2\rangle=\GammaA$ with holonomy group $\Phi_8\cong D(4)$.

Note that every element of $\Pi_8(\bfk,\bfm)$ is of the form $\bfa^\bfx\beta^z\alpha^w$ with $w\in\{0,1,2,3\}$ and

\begin{itemize}
\item $\alpha^w \beta^{z} = \bfa^* \beta^{(-1)^w z} \alpha^{(-1)^z w}$
\item if $w=0,2$, then
\begin{align*}
(\bfa^{\bfx}\beta^z\alpha^w)^{k}&=
\begin{cases}
\bfa^* \beta^{kz}, &\text{when $k$ is even};\\
\bfa^* \beta^{kz}\alpha^w, &\text{when $k$ is odd},
\end{cases}
\end{align*}

\item if $w=1,3$, then
\begin{align*}
(\bfa^{\bfx}\beta^z\alpha^w)^{k}&=
\begin{cases}
\bfa^*, &\text{when $k\equiv 0\!\!\!\pmod4$};\\
\bfa^*\beta^{z}\alpha^w, &\text{when $k\equiv 1\!\!\!\pmod4$};\\
\bfa^*\alpha^{\{1+(-1)^z\}w}, &\text{when $k\equiv 2\!\!\!\pmod4$};\\
\bfa^*\beta^{z}\alpha^{\{2+(-1)^z\}w}, &\text{when $k\equiv 3\!\!\!\pmod4.$}
\end{cases}
\end{align*}
\end{itemize}

\begin{lemma}
The subgroup ${\Gamma_{\!A}}=\langle a_1, a_2, t\rangle$ is a
a characteristic subgroup of $\Pi_8(\bfk,\bfm)$.
\end{lemma}

\begin{proof}
Let $\varphi:\Pi_8(\bfk,\bfm)\to\Pi_8(\bfk,\bfm)$ be an automorphism.
Using the relation $[a_1,a_2]=1$ and using the fact that the elements of the form $\bfa^\bfx\alpha^2$ are torsion elements of order $2$, we can derive the following possibilities:
\begin{enumerate}
\item $\varphi(a_i)=\bfa^{\bfn_{i}}\beta^{z_i}$,
\item $\varphi(a_1)=\bfa^{\bfn_{1}}, \,\,\varphi(a_2)=\bfa^{\bfn_2}\beta^z\alpha^{w}$ with $w\in\{1,2,3\}$,
\item $\varphi(a_1)=\bfa^{\bfn_{1}}\beta^{z_1}, \,\, \varphi(a_2)=\bfa^{\bfn_2}\beta^{z_2}\alpha^{2}$ with $z_i\ne0$,
\item $\varphi(a_1)=\bfa^{\bfn_{1}}\beta^z\alpha^{w}, \varphi(a_2)=\bfa^{\bfn_2}$ with $w\in\{1,2,3\}$,
\item $\varphi(a_1)=\bfa^{\bfn_{1}}\beta^{z_1}\alpha^{2}, \varphi(a_2)=\bfa^{\bfn_2}\beta^{z_2}$ with $z_i\ne0$,
\item $\varphi(a_i)=\bfa^{\bfn_{i}}\beta^{z_i}\alpha^{2}$ with $z_i\ne0$,
\item $\varphi(a_1)=\bfa^{\bfn_{1}}\beta^{z}\alpha^{w_1}, \varphi(a_2)=\bfa^{\bfn_2}\beta^{z}\alpha^{w_2}$ with $w_i\in\{1,3\}$.
\end{enumerate}

Consider the possibility (2):
$\varphi(a_1)=\bfa^{\bfn_{1}}, \,\,\varphi(a_2)=\bfa^{\bfn_2}\beta^z\alpha^{w}$
with $w=1,2,3$.
Let $\varphi(\beta)=\bfa^\bfx\beta^u\alpha^v$. Then
$\beta a_1\beta^{-1}=N(a_1)$ induces that
$$
\bfa^*=\bfa^*(\beta^z\alpha^w)^{\ell_{21}'}.
$$
If $w=2$ then $\ell_{21}'$ must be even and so $z=0$ as $\ell_{21}'\ne0$. Since $\ell_{22}'$ is odd, $\varphi(a_2)=\bfa^{\bfn_2}\alpha^2$. This element is of order $2$, which is impossible. Hence $w$ is odd. We can show in a similar way that $\ell_{21}'\equiv0\pmod4$ or else $\ell_{21}'\equiv2\pmod4$ and $z$ is odd; in any case since $\ell_{21}'$ is even, both $\ell_{11}'$ and $\ell_{22}'$ are odd.
Note that $\beta a_2\beta^{-1}=a_1^{\ell_{21}'}a_2^{\ell_{22}'}$ induces
\begin{align*}
&\bfa^*\beta^{2u+(-1)^vz}\alpha^{(-1)^u(w+((-1)^z-1)v}\\
&=
\begin{cases}
\bfa^*\beta^z\alpha^{3w}&\text{when $\ell_{22}'\equiv3\hspace{-.3cm}\pmod4$ and
$z$ is even};\\
\bfa^*\beta^z\alpha^w&\text{otherwise.}
\end{cases}
\end{align*}
This shows that
if $v$ is even then $u=0$, and if $v$ is odd then $u=z$. Hence $\varphi(\beta)$ is of the form $\bfa^\bfx\alpha^2$ or $\bfa^\bfx\beta^z\alpha^w$ with $w$ odd. The element of the form $\bfa^\bfx\alpha^2$ is of order $2$ and so $\varphi(\beta)=\bfa^\bfx\beta^z\alpha^w$. Note also that $\varphi(\beta^2)=\bfa^*$ or $\bfa^*\alpha^2$. By Remark~\ref{FI}, $\Pi_8(\bfk,\bfm)$ has a fully invariant subgroup $\Lambda\subset\GammaA=\langle a_1, a_2,\beta^2\rangle$ and hence $\varphi$ induces an automorphism $\bar\varphi$ on the group $\Lambda/\Lambda\cap\langle a_1,a_2\rangle\cong\bbz$, which is generated by some even power of $\bar\beta$, say $\bar\beta^{2k}$. Thus $\bar\varphi(\bar\beta^{2k})=\bar\beta^{\pm 2k}$. This implies that $\varphi(\beta^{2k})=\bfa^*\beta^{\pm 2k}$. This contradicts that $\varphi(\beta^2)=\bfa^*$ or $\bfa^*\alpha^2$. Thus the possibility (2) cannot occur.
By a symmetry of (2) and (4), the possibility (4) also can be eliminated.

Consider the possibility (3):
$\varphi(a_1)=\bfa^{\bfn_{1}}\beta^{z_1}, \,\, \varphi(a_2)=\bfa^{\bfn_2}\beta^{z_2}\alpha^{2}$
with $z_1\ne 0$.
Let $\varphi(\beta)=\bfa^\bfx\beta^u\alpha^v$. Then
$\beta a_i\beta^{-1}=N(a_i)$ induces that $z_1\ell_{1i}'+z_2\ell_{2i}'=(-1)^v z_i$.
Thus $(z_1,z_2)$ is a solution of $(N^t-(-1)^vI)\bfx={\bf0}$.
However, since $N^t$ has irrational eigenvalues, it follows that $z_i=0$, a contradiction.
Similarly, we can show that the possibility (5) cannot occur.

Consider the possibility (6):
$\varphi(a_i)=\bfa^{\bfn_{i}}\beta^{z_i}\alpha^{2}.$
Let $\varphi(\beta)=\bfa^\bfx\beta^u\alpha^v$.
Then by the same reason as above, $\beta a_i\beta^{-1}=N(a_i)$ induces that $z_1\ell_{1i}'+z_2\ell_{2i}'=(-1)^v z_i$ and so $z_i=0$, a contradiction.

Consider the possibility (7):
$\varphi(a_i)=\bfa^{\bfn_{i}}\beta^{z}\alpha^{w_i}$
with $w_i$ odd.
Let $\varphi(\beta)=\bfa^\bfx\beta^u\alpha^v$. Observe that
\begin{align*}
\varphi(\beta a_i\beta^{-1})=\bfa^*\beta^{2u+(-1)^vz}\alpha^{(-1)^u((-1)^z-1)v+w_i)}.
\end{align*}
If $2u+(-1)^vz=0$ then $z$ is even and so $(-1)^u((-1)^z-1)v+w_i)=(-1)^uw_i$ is odd. On the other hand, $\varphi(N(a_i))=\bfa^*(\beta^z\alpha^{w_1})^{\ell_{1i}'}(\beta^z\alpha^{w_2})^{\ell_{2i}'}$. When {$\ell_{1i}'+\ell_{2i}'$ is even}, it can be seen that $\varphi(N(a_i))=\bfa^*\alpha^{\text{even}}$. By comparing both sides, we obtain a contradiction. When {$\ell_{1i}'+\ell_{2i}'$ is odd}, it can be seen that $\varphi(N(a_i))=\bfa^*\beta^z\alpha^{\text{odd}}$. By comparing both sides, we obtain that $z=2u+(-1)^vz$. 
This implies that if $v$ is even then $u=0$ and if $v$ is odd then $u=z$. Consequently, $\varphi(\beta)=\bfa^{\bfx} \alpha^2$ or $\bfa^{\bfx}\beta^z \alpha^w$ with $w$ odd. Since $\bfa^\bfx\alpha^2$ is of order $2$, this case is eliminated. By the same reason as used in the possibility (2), we also can exclude the remaining case where $\varphi(\beta)=\bfa^{\bfx}\beta^z \alpha^w$ with $w$ odd.

\bigskip
In all, we have the only possibility (1): $\varphi(a_i)=\bfa^{\bfn_{i}}\beta^{z_i}$.
From Remark~\ref{FI}, since $\langle \bfa^{\bfm_1}, \bfa^{\bfm_2}\rangle$ is a finite index subgroup of $\langle a_1,a_2\rangle\cong\bbz^2$, we have $\det[\bfm_1\ \bfm_2]\ne0$. Furthermore, $\langle \bfa^{\bfm_1}, \bfa^{\bfm_2}\rangle$ is fully invariant, and so $\varphi(\bfa^{\bfm_i})\in \langle \bfa^{\bfm_1}, \bfa^{\bfm_2}\rangle$. Since
\begin{align*}
\varphi(\bfa^{\bfm_i})&=\varphi(a_1)^{m_{1i}}\varphi(a_2)^{m_{2i}}
=(\bfa^{\bfn_1}\beta^{z_1})^{m_{1i}}(\bfa^{\bfn_{2}}\beta^{z_2})^{m_{2i}}\\
&=\bfa^*\beta^{m_{1i}z_1+m_{2i}z_2},
\end{align*}
we have $m_{1i}z_1+m_{2i}z_2=0$. As the matrix $[\bfm_1\ \bfm_2]$ is nonsigular, $z_1=z_2=0$ and hence $\varphi(a_i)=\bfa^{\bfn_{i}}$.

Since $\alpha^4=1$, we have $\varphi(\alpha)=\bfa^{\bfx}\beta^z\alpha^{w}$
with $z$ even and $w\in\{1,3\}$. Thus
the subgroup $\langle a_1,a_2,\alpha^2 \rangle$ is a a characteristic subgroup of $\Pi_8(\bfk,\bfm)$, so $\varphi$ induces an automorphism $\bar\varphi$ on the quotient group $\Pi_8(\bfk,\bfm)/\langle a_1,a_2, \alpha^2 \rangle\cong\bbz\rtimes\bbz_2$, with generators $\bar\beta,\bar\alpha$. Hence we must have $\bar\varphi(\bar\beta)=\bar\beta^{\pm1}$ or $\varphi(\beta)=\bfa^\bfq\beta^{\pm1}\alpha^{2v}$.
Consequently, the subgroups $\GammaA=\langle a_1,a_2,\beta^2\rangle$ is a characteristic subgroup of $\Pi_8(\bfk,\bfm)$.
\end{proof}

\begin{theorem}
The SC-groups $\Pi_8(\bfk,\bfm)$ have the $R_\infty$ property.
\end{theorem}

\begin{proof}
Let $\varphi$ be an automorphism on $\Pi_8(\bfk,\bfm)$ and let $\varphi'$ be the restriction of $\varphi$ on ${\Gamma_{\!A}}$. Since $[\Pi_8(\bfk,\bfm):\Gamma_{\!A}]=8$, by Theorem~\ref{HLP}, we have
$$
R(\varphi)\ge\frac{1}{8}\sum_{g\in\Pi_8(\bfk,\bfm)/\Gamma_{\!\!A}}R(\tau_{g}\varphi').
$$
Since
$$
\tau_\alpha\varphi(a_i)=\alpha(\bfa^{\bfn_i})\alpha^{-1}=\bfa^{M\bfn_i},\
\tau_\alpha\varphi(t^{\pm1})=\bfa^*t^{\mp1},
$$
it shows that $\varphi'$ is of type (II) if and only if $\tau_\alpha\varphi'$ is of type (I). Hence the theorem follows from Theorem~\ref{GA}.
\end{proof}

\smallskip

\noindent
\textbf{Acknowledgments.}
{The authors would like to thank the referee for making careful corrections to a few expressions.}

\end{document}